\documentclass[10 pt]{amsart}
\usepackage{amsmath}
\usepackage{latexsym,amsfonts,amssymb,mathrsfs}
\usepackage{mathdots}
\usepackage{color}
\usepackage [all,2cell,color]{xy}
\usepackage{enumitem} %allows e.g. to change labels in enumerations
\usepackage{ dsfont } % for bold 1

\usepackage[
textwidth=3cm,
textsize=small,
colorinlistoftodos]
{todonotes}

\usepackage{url}
\usepackage{enumitem} % to customize itemize lists

\usetikzlibrary{positioning} %should allow to draw more precise arrows
\usetikzlibrary{decorations.pathreplacing}%need this to draw nice curly brackets
\usetikzlibrary{arrows, decorations.markings} %need this for nice arrows in tikz pictures, markings in particular for double arrows

\usetikzlibrary{backgrounds,shapes,arrows,calc} %needed to highlight parts of a picture, as well as the next definitions:

\tikzset{mapstikz/.style={-stealth, 
decoration={markings,mark=at position 0pt with {\arrow[scale=0.5]{|}}}, preaction={decorate}}}%this is the style for mapsto in tikz

\usetikzlibrary{cd}

\usepackage[twoside=false,left=3cm,right=3cm,top=3cm,bottom=3cm]{geometry}

\theoremstyle{plain}   % This is the default, anyway
\newtheorem{thm}{Theorem}[section] % numbered theorem
\makeatletter\let\c@thm\c@thm\makeatother

\makeatletter\let\c@cor\c@thm\makeatother

\makeatletter\let\c@lem\c@thm\makeatother
\newtheorem{prop}{Proposition}[section]
\makeatletter\let\c@prop\c@thm\makeatother

\makeatletter\let\c@claim\c@thm\makeatother
\newtheorem{conjecture}{Conjecture}[section]
\makeatletter\let\c@conjecture\c@thm\makeatother
\newtheorem{consequence}{Consequence}[section]
\makeatletter\let\c@consequence\c@thm\makeatother

\makeatletter\let\c@question\c@thm\makeatother

\newtheorem*{unnumberedtheoremA}{Theorem A}
\newtheorem*{unnumberedtheoremB}{Theorem B}
\newtheorem*{unnumberedtheoremC}{Theorem C}
\newtheorem*{unnumberedtheoremD}{Theorem D}

\theoremstyle{definition}

\newtheorem{defn}{Definition}[section]
\makeatletter\let\c@defn\c@thm\makeatother

\makeatletter\let\c@const\c@thm\makeatother
\newtheorem{notn}{Notation}[section]
\makeatletter\let\c@notn\c@thm\makeatother

\theoremstyle{remark}

\newtheorem{rmk}{Remark}[section]
\makeatletter\let\c@rmk\c@thm\makeatother
\newtheorem{ex}{Example}[section]
\makeatletter\let\c@ex\c@thm\makeatother

\makeatletter\let\c@observationn\c@thm\makeatother
\newtheorem{digression}{Digression}[section]
\makeatletter\let\c@digression\c@thm\makeatother

\makeatletter
\let\c@equation\c@thm
\numberwithin{equation}{section}
\makeatother

\usepackage{hyperref}

\usepackage[capitalise]{cleveref}

\newcommand{\newrefformat}[2]{}

\crefname{lem}{Lemma}{Lemmas}
\crefname{thm}{Theorem}{Theorems}
\crefname{defn}{Definition}{Definitions}
\crefname{notn}{Notation}{Notations}
\crefname{const}{Construction}{Constructions}
\crefname{prop}{Proposition}{Propositions}
\crefname{conjecture}{Conjecture}{Conjectures}
\crefname{consequence}{Consequence}{Consequences}
\crefname{rmk}{Remark}{Remarks}
\crefname{digression}{Digression}{Digressions}
\crefname{cor}{Corollary}{Corollaries}
\crefname{equation}{Display}{Displays}
\crefname{ex}{Example}{Examples}

\newcommand{\cC}{\mathcal{C}}

\newcommand{\cI}{\mathcal{I}}
\newcommand{\cJ}{\mathcal{J}}
\newcommand{\cK}{\mathcal{K}}

\newcommand{\cM}{\mathcal{M}}

\newcommand{\cQ}{\mathcal{Q}}

\newcommand{\cS}{\mathcal{S}}

\newcommand{\cV}{\mathcal{V}}

\newcommand{\cat}{\cC\!\mathit{at}}

\newcommand{\set}{\cS\!\mathit{et}}

\newcommand{\kan}{\cK\!\mathit{an}}

\newcommand{\sset}{\mathit{s}\set}

\newcommand{\scat}{\mathit{s}\cat}

\newcommand{\qcat}{\mathit{q}\cat}

\newcommand{\fatslice}{\mathbin{\mkern-1mu{/}\mkern-5mu{/}\mkern-1mu}}

\DeclareMathOperator{\Ob}{Ob}
\DeclareMathOperator{\Tw}{Tw}

\DeclareMathOperator{\Nec}{Nec}
\DeclareMathOperator{\eq}{eq}

\DeclareMathOperator{\holim}{holim}
\DeclareMathOperator{\hofib}{hofib}
\DeclareMathOperator{\fib}{fib}
\DeclareMathOperator{\id}{id}

\DeclareMathOperator{\Mor}{Mor}

\newcommand{\Fun}{\textnormal{Fun}}

\DeclareMathOperator{\Hom}{Hom}
\DeclareMathOperator{\Map}{Map}

\DeclareMathOperator{\op}{op}

\usepackage{pigpen}
\newcommand{\po}{\ar@{}[dr]|{\text{\pigpenfont R}}}
\newcommand{\pb}{\ar@{}[dr]|{\text{\pigpenfont J}}}

\newcommand{\nocontentsline}[3]{}
\newcommand{\tocless}[2]{\bgroup\let\addcontentsline=\nocontentsline#1{#2}\egroup}
\author{Martina Rovelli}
\address{Department of Mathematics,
Johns Hopkins University,
Baltimore, USA
}
\email{mrovelli@math.jhu.edu}

\thanks{
The author was partially funded by the Swiss National Science Foundation, grant P2ELP2\textunderscore172086.}

\begin{document}

\title{Weighted limits in an $(\infty,1)$-category}

\maketitle

\begin{abstract}
We introduce the notion of weighted limit in an arbitrary quasi-category, suitably generalizing ordinary limits in a quasi-cat\-egory, and classical weighted limits in an ordinary category. This is accomplished by generalizing Joyal's approach: we identify a meaningful construction for the quasi-category of weighted cones over a diagram in a quasi-category, whose terminal object is the weighted limit of the considered diagram. When the quasi-category arises as the homotopy coherent nerve of a category enriched over Kan complexes, we use techniques by Riehl-Verity to show that the weighted limit agrees with the homotopy weighted limit in the sense of enriched category theory, for which explicit constructions are available. When the quasi-category is complete, tensored and cotensored over the quasi-category of spaces, we discuss a possible comparison of our definition of weighted limit with the approach by Gepner-Haugseng-Nikolaus.
\end{abstract}

\tableofcontents

\section*{Introduction}

In ordinary category theory, a limit is a universal object with a cone over a given diagram; the cone is given by maps from the limit to all the objects in the diagram that commute with all maps in the diagram. But especially in enriched or higher categorical contexts, it is desirable to generalize this notion of limit to allow for more exotic shapes.
Formally, the desired shape of a cone is specified by means of something called the \emph{weight}, this being another diagram defined on the same indexing category. The universal object with a cone of that shape is then refereed to as the \emph{weighted limit}.

Weighted limits are ubiquitous, though seldom recognized as such. Most famously homotopy limits and colimits as constructed by Bousfield-Kan \cite{BK} are simply weighted limits defined for a particular simplicially enriched weight. Key constructions in $2$-category theory such as the arrow category, the comma category, and the category of monads are also weighted limits, this time for a particular categorically enriched weight \cite{streetlimits}.
In unenriched category theory, the matching and latching objects are understood in \cite{RVreedy} as weighted limits and colimits with very natural weights valued in sets.

Several objects of interest in brave new algebra, such as the quasi-category of genuine $G$-spectra for a compact Lie group $G$ from \cite{AMGR}, satisfies the universal property of a suitably derived weighted limit in the $(\infty,2)$-category of $(\infty,1)$-categories.
These ad hoc constructions of weighted limits have been used to great effect and provide a strong motivation towards developing a fully general theory of weighted limits in an arbitrary $(\infty,2)$-category, rather than just this base case.

With the intent of providing foundations for such a theory,  we start by understanding and developing the theory of weighted limits in an arbitrary $(\infty,1)$-category instead.
The purpose of this paper is to define, construct and study the properties of weighted limits in a generic $(\infty,1)$-category. 
 We allow the $(\infty,1)$-category to be presented by different models and show that the notion of weighted limits is suitably compatible with the change of model, complementing the analogous analysis of (unweighted) limits in $(\infty,1)$-categories by Riehl-Verity \cite{RV7}.
We focus on the models of $\kan$-enriched categories and quasi-categories, but the approach taken for the latter easily generalizes to other models, such as complete Segal spaces. For quasi-categories that happen to be complete, tensored and cotensored over the $\infty$-category of spaces, we discuss how our constructions should compare with those by Gepner-Haugseng-Nikolaus \cite{GHN}. Weighted limits in an $(\infty,2)$-category will then be treated in a forthcoming paper.

\addtocontents{toc}{\protect\setcounter{tocdepth}{1}}
\subsection*{Weighted limits in enriched category theory}
%\addtocontents{toc}{\protect\setcounter{tocdepth}{1}}
Weighted limits were introduced as the correct general notion of limit for enriched categories independently by Auderset \cite{auderset}, by Borceux-Kelly \cite{borceuxkelly}, and by Street-Walters in unpublished work.

Given a $\cV$-enriched category $\cQ$, the weight for a diagram $D\colon\cJ\to\cQ$ has the form of a $\cV$-enriched functor
$W\colon\cJ\to\cV$
and its role is to encode the ``shape'' of these generalized cones. 
Indeed, while an ordinary cone from an object $A$ to the diagram $D$ picks a single arrow $A\to Dj$ for any $j$ in $\cJ$, a weighted cone consists of a collection of maps encoded by a morphism $W(j)\to\Map_{\cQ}(A,Dj)$ in $\cV$.

More precisely, the limit of the diagram $D$ weighted by $W$ is an object $\lim^WD$ of $\cQ$ for which there is a natural isomorphism in $\cV$
$$\Map_{\cQ}(A,\lim{}^WD )\cong\Map_{\cV\Fun(\cJ,\cV)}(W,\Map_{\cQ}(A,D-)).$$
When the weight is the functor $\ast_{\cJ}\colon\cJ\to\cV$ constant at the monoidal unit $\ast$ of $\cV$, weighted cones are ordinary cones, and the expression simplifies to
$$\Map_{\cQ}(A,\lim{}^{\ast_{\cJ}}D )\cong\Map_{\cV\Fun(\cJ,\cQ)}(\Delta A,D),$$
so that limits weighted by the constant weight have a more familiar universal property. We refer to these limits as ``unweighted limits'', though they appear in the literature as ``conical'' or ``enriched'' limits.

\addtocontents{toc}{\protect\setcounter{tocdepth}{1}}
\subsection*{Homotopy weighted limits in enriched category theory}
When the base for enrichment $\cV$ and the category $\cQ$ are endowed with homotopical structures,
the categories of enriched functors $\cV\Fun(\cJ,\cQ)$ and $\cV\Fun(\cJ,\cV)$ naturally inherit a homotopical structure in which the weak equivalences are detected levelwise.

Thanks to these homotopical structures, it makes sense to ``derive'' the universal property that defines weighted limits,
by replacing mapping objects with derived mapping objects.
The corresponding universal property identifies (up to equivalence) a weighted limit that is homotopically meaningful, and should be thought of as a ``homotopy weighted limit''.
In other words,
the weighted homotopy limit of a diagram $D\colon\cJ\to\cQ$ consists of an object $\holim^WD$ of $\cQ$ together with natural equivalences in $\cV$
$$\Map^h_{\cQ}(A,\holim^WD)\simeq\Map^h_{\cV\Fun(\cJ,\cV)}(W,\Map^h_{\cQ}(A,D-)).$$

The framework in which this idea has been explored the most is when $\cQ$ arises as the category of fibrant objects of a $\cV$-model structure $\cM$. In this case,
for the mapping spaces to have the correct homotopy type and be automatically derived it is enough
%to focus on diagrams $D\colon\cJ\to\cQ$, which are necessarily projectively fibrant, and
to take a projectively cofibrant replacement $W^{cof}$ for the weight $W\colon\cJ\to\cV$.
It follows that the weighted homotopy limit, $\holim^WD$, can be realized as the strict weighted limit, $\lim{}^{W^{cof}}D$, or even as $\lim{}^WD$ if the weight is already cofibrant.

When the base for the enrichment $\cV$ is $\cat$, endowed with the canonical model structure, homotopy weighted limits in the (model) category $\cat$ are discussed by Kelly in \cite{kelly2limits} under the name of ``weighted bilimits'', and a weight $W\colon\cJ\to\cat$ is projectively cofibrant when it is ``flexible''.

When the base for the enrichment $\cV$ is $\sset$, endowed with the Kan-Quillen model structure, 
homotopy weighted limits are discussed by Riehl and Verity in \cite{RiehlCHT,RV7}.
In this setting, a theorem of Gambino
provides a different viewpoint on homotopy weighted limits, asserting that, when $\cM$ is a simplicial 
model category, the weighted limit defines a Quillen bifunctor
$$\lim{}^--\colon s\Fun(\cJ,\cV)_{proj}^{\op}\times s\Fun(\cJ,\cM)_{proj}\to\cM,$$
and the homotopy weighted limit $\holim^--$ is precisely its derived functor.

In a similar framework to that of \cite{RV7}, we discuss the theory of homotopy weighted limits in any $(\infty,1)$-category $\cQ$ presented by a minimal structure, namely just that of a $\kan$-category. This enrichment naturally induces a notion of weak equivalence in $\cQ$, given by those maps that induce weak equivalences on hom Kan complexes in both variables. This class of weak equivalences make $\cQ$ into a homotopical category.

We show that, under suitable completeness conditions for $\cQ$, the homotopy limit of a diagram $D\colon\cJ\to\cQ$ weighted by $W\colon\cJ\to\sset$ can still be computed as the strict limit $\lim{}^{W^{cof}}D$.
We also prove the following analog of Gambino's result, which will appear as \cref{weightedhomotopylimitsarederivedfunctors}.
\begin{unnumberedtheoremA}
If $\cQ$ is a complete simplicial $\kan$-category that is cotensored over $\sset$,
the homotopy weighted limit $\holim^--$ is the derived functor of the strict weighted limit
$$\lim{}^--\colon s\Fun(\cJ,\sset)^{\op}\times s\Fun(\cJ,\cQ)\to\cQ,$$
when both $s\Fun(\cJ,\sset)$ and $s\Fun(\cJ,\cQ)$ are endowed with the levelwise homotopical structure.
\end{unnumberedtheoremA}

In particular, the homotopy limit $\holim^{\ast_{\cJ}}D$ weighted by the constant weight $\ast_{\cJ}\colon\cJ\to\sset$ can be computed as a strict limit $\lim{}^{\ast_{\cJ}^{cof}}D$ weighted by a cofibrant replacement $\ast_{\cJ}^{cof}$ of the constant weight, and agrees with the homotopy limit $\holim D$ in a simplicial category
intended in the classical sense, as e.g.~ in \cite{BK}.

The classical formulas to compute homotopy (co)limits share the same formalism. Indeed, they compute $\holim D$ as a strict (co)limit of $D$ weighted by suitable weights, which we can a posteriori recognize as cofibrant replacements of the constant weight.
When $\cJ$ is an ordinary unenriched category, a cofibrant replacement for the constant weight is the functor $N(J_{/_-})\colon\cJ\to\sset$, which was first used by Bousfield-Kan in \cite{BK} to compute homotopy (co)limits of diagrams $\cJ\to\cQ$.
When $\cJ=\mathfrak C[J]$ is the homotopy coherent realization of a simplicial set $J$, a cofibrant replacement for the constant weight is the functor $\Map_{\mathfrak C[\Delta[0]\star J]}(0,-)\colon\mathfrak C[J]\to\sset$, which was more recently used by Riehl-Verity in \cite{RV7} to compute homotopy (co)limits of homotopy coherent diagrams $\mathfrak C[J]\to\cQ$.

\addtocontents{toc}{\protect\setcounter{tocdepth}{1}}
\subsection*{Limits in quasi-categories}
Given any $\kan$-category $\cQ$, the enrichment induces a canonical homotopical structure on $\cQ$, in which the weak equivalences are tested on hom spaces in both variables.
As a special case of the machinery that we just described, it thus makes sense to talk about homotopy weighted limits in any $\kan$-enriched category, which is an established model for an $(\infty,1)$-category. 
However, given that much of $(\infty,1)$-category theory has been developed for the model of quasi-categories, we  want to make sense of (homotopy) weighted limits for an $(\infty,1)$-category presented by a quasi-category\footnote{The nature of the model of quasi-categories as $(\infty,1)$-categories is such that strict notions do not make sense, and all constructions are automatically derived. For instance, the only kind of limit that makes sense in a quasi-category is already a homotopy limit. For the same reason, it only makes sense to introduce \emph{homotopy} weighted limits in a quasi-category, and we will just refer to them as weighted limits}.

For such a formalization to be meaningful, it should suitably generalize the well-established notion of (unweighted) limit in a quasi-category from \cite{joyalquasicategories}. Given a diagram $d\colon J\to Q$ in a quasi-category $Q$, the limit is defined as the terminal object in a \emph{quasi-category} of cones over $d$.
Joyal \cite{joyalnotes} gives two equivalent constructions for such a quasi-category, in the form of a \emph{neat slice} $Q_{/d}$ and a \emph{fat slice} $Q_{\fatslice d}$ over the diagram $d$. The slices functors
$$(-)_{/(-)}\ ,\ (-)_{\fatslice (-)}\ \colon{}^{J_/}\sset\to\sset$$
are in turn defined as right adjoints of a neat and a fat \emph{join construction}:
$$-\star J\ ,\ -\diamond J\ \colon\sset\to{}^{J_/}\sset.$$
All these constructions will be adapted to the weighted context.

\addtocontents{toc}{\protect\setcounter{tocdepth}{1}}
\subsection*{Weighted limits in quasi-categories}

In order to define weighted limits for a diagram $d\colon J\to Q$ in a quasi-category $Q$, we must first decide what the nature of the weight should be.
Given that any quasi-category $Q$ is morally enriched over the $(\infty,1)$-category $Kan$ of spaces, the natural guess would be to look for an assignment $w\colon J\to Kan$ of some kind.
However, in the quasi-categorical model it is laborious to specify such a homotopy coherent diagram. Rather than trying to formalize the nature of the assignment $w$, it is easier to instead encode it into a map of simplicial sets $p\colon\tilde J\to J$, in which each value $w(j)$ is recorded by the fiber $p^{-1}(j)$ over $j$.

Given a weight $p\colon\tilde J\to J$, in \cref{weightedconesqcat,weightedconesqcatfat} we define weighted versions of the neat and the fat join constructions,
$$-\star^p J\ ,\ -\diamond^p J\ \colon\sset\to{}^{J_/}\sset$$
by taking a pushout of the unweighed join constructions with $\tilde J$ along the fibration $p$.
The right adjoints of these functors,
$$(-)^p_{/(-)}\ ,\ (-)^p_{\fatslice (-)}\ \colon{}^{J_/}\sset\to\sset,$$
are then weighted versions of the neat and fat slice constructions.
In harmony with the unweighted picture, the neat weighted slice $Q^p_{/d}$ and the fat weighted slice $Q^p_{\fatslice d}$ are shown in \cref{Comparison of models} to be equivalent.

\begin{unnumberedtheoremB}
For any weight $p\colon\tilde J\to J$, the neat weighted slice $Q^p_{/d}$ and fat weighted slice $Q^p_{\fatslice d}$ in a quasi-category $Q$ over a diagram $d\colon J\to Q$ are equivalent models for a quasi-category of weighted cones in $Q$ over the diagram $d$, i.e.
$$Q^p_{/d}\simeq Q^p_{\fatslice d}.$$
\end{unnumberedtheoremB}

Relying on the construction of the quasi-category of weighted cones given in Theorem B, we then define the weighted limit to be their terminal object, which can be characterized by an explicit lifting property.

The fact that this definition is the correct one is validated by showing in \cref{Comparison of weighted limits in different models} that the notion of homotopy weighted limit in a $\kan$-enriched category
$\cQ$ is acually compatible with the one we just introduced for weighted limits in its homotopy coherent nerve $\mathfrak N[\cQ]$. The method is inspired by the analogous result for unweighted limits presented in \cite{RV7}.

To make this precise, recall that Cordier's homotopy coherent nerve construction $\mathfrak N$ from \cite{cordier} forms together with the homotopy realization an adjoint pair $\mathfrak C\colon\sset\rightleftarrows\scat\colon\mathfrak N$.
In particular, any simplicial diagram $D\colon\mathfrak C[J]\to\cQ$ in $\cQ$ adjoins to a diagram $d\colon J\to\mathfrak N[\cQ]$ in the homotopy coherent nerve $\mathfrak N[\cQ]$. Also, any weight $W\colon\mathfrak C[J]\to\sset$ can be turned into a fibration $Un(W)\colon\tilde J\to J$ using Lurie's unstraightening machinery \cite{htt}.

\begin{unnumberedtheoremC}
Let $\cQ$ be a $\kan$-category, $D\colon\mathfrak C[J]\to\cQ$ a diagram in $\cQ$ and $W\colon\mathfrak C[J]\to\kan$ a weight.
If $\holim^{W}D$ exists then $\lim{}^{Un(W)}d$ exists and
$$\lim{}^{Un(W)}d\simeq\holim^{W}D.$$
\end{unnumberedtheoremC}

The theorem, which will appear as \cref{maintheorem}, provides evidence to the fact that the definition of weighted limits in quasi-categories is meaningful.
Furthermore, it also implies that under completeness assumptions for $\cQ$ every weighted limit of a diagram in the quasi-category $\mathfrak N[\cQ]$ can be realized as a weighted limit in $\cQ$, for which explicit formulas are available.

Interestingly enough, the construction of weighted limits in quasi-categories also allows us to prove in \cref{allweightedhomototopylimitsareconical} that every weighted limit in an $(\infty,1)$-category can be computed as the (unweighted) limit of a suitably fatter diagram, generalizing a classical result from strict $1$-category theory.

\begin{unnumberedtheoremD}
Given a quasi-category $Q$, the limit of a diagram $J\to Q$ weighted by $p\colon\tilde J\to J$ is the limit of the diagram
$d\circ p\colon\tilde J\to J\to Q$, i.e,
$$\lim{}^pd\simeq\lim(d\circ p).$$
\end{unnumberedtheoremD}

Finally, one could take an alternative approach on weighted limits in higher categories making use of the machinery of presentable quasi-categories developed by Lurie \cite{htt,higheralgebra}.
When a quasi-category $Q$ is presentable, and also tensored and cotensored over the $(\infty,1)$-category $Kan$ of spaces, it can be understood as forming part of a two-variable \emph{tensor-cotensor-hom adjunction}. With this additional structure, it is possible to define weighted limits by the quasi-categorical analog of the usual end formula, and this is the approach taken by Gepner-Haugseng-Nikolaus in \cite[\textsection2]{GHN}.

A comparison between their approach and ours requires facility with the theory of two-variable adjunctions of quasi-categories and we are not certan whether quasi-categorical analogs of all the $1$-categorical results we would need are in the literature.
Assuming \cref{comparisonGHN,sliceforcotensored}, two quasi-categorical generalizations of classical properties that encode the universal property of the tensor cotensor-hom adjunction and the relation with the quasi-category of weighted cones $Q^p_{\fatslice d}$, we can easily show as  \cref{weightedlimitcotensored} that the Gepner-Haugseng-Nikolaus definition for the notion of weighted limit gives an alternative formula for the notion of weighted limit that appears here.

\addtocontents{toc}{\protect\setcounter{tocdepth}{1}}
\subsection*{Further work}
In forthcoming work, we enhance the techniques from this paper with the aim of defining and studying weighted limits in arbitrary $(\infty,2)$-categories. We will work with Riehl-Verity's \emph{$2$-complicial sets}, which are the model of $(\infty,2)$-categories that is closest to quasi-categories. The identification of suitable  neat and fat weighted join constructions in the new framework allows us to construct a neat and a fat weighted slice construction over a diagram $d\colon J\to K$ in a $2$-complicial set $K$. These constructions are designed and shown to be equivalent models for the $(\infty,2)$-category of weighted cones over $d$, so it makes sense to take the weighted limit of $d$ to be their terminal object. Having understood the correct notion of weighted limits in $2$-complicial sets, we can then proceed to exporting it to different models of $(\infty,2)$-categories and studying how they compare to existing instances of weighted limits in $(\infty,2)$-categories present in the literature.

\addtocontents{toc}{\protect\setcounter{tocdepth}{1}}
\subsection*{Acknowledgements}
The author is grateful to Emily Riehl for suggesting this project and for sharing some of her work in progress with Dominic Verity, as well as for useful discussions on the subject and valuable feedback.
This paper also benefited from conversations with Viktoriya Ozornova.

\section{The (quasi-)category of weighted cones}
\label{The (quasi-)category of weighted cones}

Given an ordinary category $\cC$, the universal property of the limit cone of a diagram $D\colon \cJ\to\cC$ is that it is terminal in the category of cones over that diagram.

The same point of view has been adapted \cite{joyalquasicategories,htt,RV5} in order to make sense of the limit of a diagram $d\colon \cJ\to\cQ$ in a quasi-category $\cQ$, in that the limit of $d$ is the terminal object in the quasi-category of cones over $d$. As we will briefly recall, Joyal \cite{joyalnotes} gave two equivalent models for the quasi-category of cones over the diagram $d$, in the form of a neat slice and in the form of a fat slice over $d$. These constructions arise as right adjoints of the neat and the fat join constructions, respectively.

By further generalizing, though in a different direction, given a suitable weight
one might want to make sense of the (quasi-)category of \emph{weighted} cones over the diagram $d$, so that the weighted limit cone of the diagram will then be defined as its terminal object (cf.~ \cref{weightedlimits}).

The aim of this section is to construct such quasi-category of weighted cones, by means of suitably generalized join constructions. For the sake of exposition, we first discuss in \cref{weightedconescat} the category of weighted cones in the case of an ordinary category, and then move on to the quasi-categorical case in \cref{weightedconesqcat,weightedconesqcatfat}.

\subsection{The category of weighted cones}
\addtocontents{toc}{\protect\setcounter{tocdepth}{2}}
\label{weightedconescat}
Recall from \cite[\textsection3.1]{joyaltheory} that the join of a category ${\cI}$ with a category a category $\cJ$ has set of objects given by
$$\Ob({\cI}\star \cJ)\cong\Ob({\cI})\amalg\Ob(\cJ),$$
and homsets as follows.
\begin{enumerate}[leftmargin=*]
\item The homset between $a\in\cI$ and $b\in\cJ$ is $\Hom_{{\cI}\star \cJ}(a,b):=\ast$.
\item The homset between $a\in\cI$ and $a'\in\cI$ is $\Hom_{{\cI}\star \cJ}(a,a'):=\Hom_{{\cI}}(a,a')$.
\item The homset between $b\in\cJ$ and $b'\in\cJ$ is $\Hom_{{\cI}\star \cJ}(b,b'):=\Hom_{\cJ}(b,b')$.
\item The homset between $b\in\cJ$ and $a\in\cI$ is $\Hom_{{\cI}\star \cJ}(a,b):=\varnothing$.
\end{enumerate}
In particular, the set of morphisms can be identified with the disjoint union
$$\Mor({\cI}\star \cJ)\cong\Mor({\cI})\amalg\Mor(\cJ)\amalg\left(\Ob({\cI})\times \Ob(\cJ)\right).$$
Composition is uniquely determined by declaring that ${\cI}$ and $\cJ$ are subcategories of ${\cI}\star \cJ$.

Roughly speaking, the join of two categories is built by putting two categories next to each other and by adding exactly one morphism from the first one to the second one.
Interesting instances of the join construction are when $\cJ=[0]$, in which case $[0]\star \cJ\cong \cJ^{\triangleleft}$ is obtained by freely adding a terminal object to $\cJ$, and when $\cI=[m]$ and $\cJ=[n]$, in which case $[m]\star[n]\cong[n+1+m]$ recovers the ordinal sum on $\Delta$.

For any category ${\cI}$, there is a canonical inclusion $\cJ\to {\cI}\star \cJ$. In fact, if ${}^{\cJ/}\cat$ denotes the slice category of $\cat$ over $\cJ$, the \emph{join} with a category $\cJ$ defines a functor
$$(-)\star \cJ\colon\cat\to{}^{\cJ/}\cat,$$
As mentioned in \cite[\textsection3.1.1]{joyaltheory}, the join construction admits a right adjoint
$$(-)_{/(-)}\colon{}^{\cJ/}\cat\to\cat$$
that assigns to a diagram $d\colon \cJ\to {\cC}$ the \emph{slice} ${\cC}_{/d}$.
The property of adjointness of
$$(-)\star \cJ\colon\cat\rightleftarrows{}^{\cJ/}\cat\colon(-)_{/(-)}$$
determines the slice construction up to isomorphism.

\begin{defn}
Given a diagram $d\colon \cJ\to {\cC}$, the \emph{category of cones} in ${\cC}$ over the diagram $d\colon \cJ\to {\cC}$ is the slice category ${\cC}_{/d}$.
In particular, the set of objects of ${\cC}_{/d}$ is given by
$$\Ob\left({\cC}_{/d}\right)\cong\Hom_{\cat}\left([0],{\cC}_{/d}\right)\cong\Hom_{{}^{\cJ/}\cat}([0]\star \cJ,{\cC}),$$
and the set of morphisms of ${\cC}_{/d}$ could be unraveled from the expression
$$\Mor\left({\cC}_{/d}\right)\cong\Hom_{\cat}\left([1],{\cC}_{/d}\right)\cong\Hom_{{}^{\cJ/}\cat}([1]\star \cJ,{\cC}).$$
\end{defn}

\begin{rmk}
Given a diagram $d\colon \cJ\to {\cC}$, an object of the category of cones in ${\cC}$ over the diagram $d\colon \cJ\to {\cC}$ is precisely a \emph{cone} over $D$ acccording to the classical terminology of MacLane \cite{maclane}, i.e., a pair $(A,\varphi\colon\Delta_A\to D)$
where $A$ is an object of ${\cC}$ and $\varphi$ is a natural transformation from the constant functor at $A$ to the functor $D$.
\end{rmk}

We now define a weighted version of the join construction, whose corresponding slice construction models a category of cones weighted by any functor $W\colon \cJ\to\set$.
As clarified by the following remark, these functors correspond precisely to discrete fibrations $P\colon\tilde \cJ\to \cJ$.

\begin{rmk}

\label{straighteningunstraighteningcats}

As explained in \cite[Theorem 2.1.2]{LR}, for any category $\cJ$ there is an equivalence of categories
$$\fib\colon Disc Fib_{\cJ}\simeq\Fun(\cJ,\set)\colon el_{\cJ}$$
between functors ${\cJ}\to\set$ and \emph{discrete} fibrations of categories, which are certain fibrations of categories.

Given any discrete fibration $P\colon\tilde {\cJ}\to {\cJ}$, the collections of fibers over the objects of $\cJ$ assemble into a functor $P^{-1}(-)\colon {\cJ}\to\set$.

The category of elements $el_{\cJ}(W)$ is a special case of the Grothendieck construction.
Its objects are pairs $(j,x\in W(j))$ and whose morphisms are arrows $f\colon j\to j'$ in $\cJ$ such that $W(f)(x)=x'$.
The category of elements comes with an obvious projection $P\colon el_{\cJ}W\to {\cJ}$.
In particular, $el_{\cJ}$ respects the terminal object, so the identity of $\cJ$ corresponds to the functor constant at the terminal category.
\end{rmk}

\begin{defn}
\label{weightedjoin}
Let $P\colon\tilde {\cJ}\to {\cJ}$ be a discrete fibration. The \emph{join} of a category ${\cI}$ \emph{weighted} by $P\colon\tilde {\cJ}\to {\cJ}$ is the category defined as follows. The set of objects is given by
$$\Ob({\cI}\star^P{\cJ})\cong\Ob({\cI})\amalg\Ob({\cJ}).$$
The homsets re defined as follows.
\begin{enumerate}[leftmargin=*]
\item The homset between $a\in\cI$ and $b\in\cJ$ is $\Hom_{{\cI}\star \cJ}(a,b):=W(j)$.
\item The homset between $a\in\cI$ and $a'\in\cI$ is $\Hom_{{\cI}\star \cJ}(a,a'):=\Hom_{{\cI}}(a,a')$.
\item The homset between $b\in\cJ$ and $b'\in\cJ$ is $\Hom_{{\cI}\star \cJ}(b,b'):=\Hom_{\cJ}(b,b')$.
\item The homset between $b\in\cJ$ and $a\in\cI$ is $\Hom_{{\cI}\star \cJ}(a,b):=\varnothing$.
\end{enumerate}
In particular, the set of morphisms of the weighted join can be written as a disjoint union
$$\Mor({\cI}\star^P{\cJ})\cong\Mor({\cI})\amalg\Mor({\cJ})\amalg\left(\Ob({\cI})\times\Ob(\tilde {\cJ})\right).$$
The composition law is determined by declaring that ${\cI}$ and ${\cJ}$ are subcategories of the weighted join ${\cI}\star^p {\cJ}$, and using the functoriality of $W$ to define composition of an arrow from $\cI$ to $\cJ$ with an arrow of $\cJ$.
\end{defn}

Roughly speaking, the join of two categories ${\cI}$ and ${\cJ}$ weighted by $p$ is built by putting two categories next to each other and by adding for any $i$ in ${\cI}$ as many arrows as element of $P^{-1}(j)$ that go from $i$ to $j$.

\begin{rmk}
\label{weightedjoinaspushout}
The
weighted join can be written as a pushout of categories
$$\cI\star^P\cJ\cong (\cI\star \tilde \cJ)\amalg^P_{\tilde \cJ}\cJ,$$
as can be seen by direct verification of the universal property of pushouts.
As a special case, by \cref{straighteningunstraighteningcats}, the join weighted by the constant functor at the terminal category concides with the classical join construction.
\end{rmk}

For any category ${\cI}$, there is a canonical inclusion
$${\cJ}\to {\cI}\star\tilde {\cJ}\to (\cI\star \tilde {\cJ})\amalg^P_{\tilde {\cJ}}{\cJ}\cong {\cI}\star^P {\cJ}.$$
In fact, the weighted join construction defines a functor
$$(-)\star^P{\cJ}\colon\cat\stackrel{-\star\tilde {\cJ}}{\longrightarrow}{}^{\tilde {\cJ}/}\cat\stackrel{-\amalg_{\tilde {\cJ}}{\cJ}}{\longrightarrow}{}^{{\cJ}/}\cat.$$

\begin{prop}Let $P\colon\tilde {\cJ}\to {\cJ}$ be a discrete fibration. 
The join weighted by $P\colon\tilde {\cJ}\to {\cJ}$
admits a right adjoint
$$(-)^P_{/(-)}\colon{}^{{\cJ}/}\cat\stackrel{(-)\circ P}{\longrightarrow}{}^{\tilde {\cJ}/}\cat
\stackrel{(-)_{/(-)}}{\longrightarrow}
\cat,$$
that assigns to a diagram $d\colon {\cJ}\to {\cC}$ what we call the \emph{weighted slice} ${\cC}^P_{/_d}$.
\end{prop}

\begin{proof}
The adjunction
$$(-)\star\tilde {\cJ}\colon\cat\rightleftarrows{}^{\tilde {\cJ}/}\cat\colon(-)_{/(-)}$$
and the (dual of the) adjunction from \cite[Proposition 9.18]{awodey}
$$-\amalg_{\tilde {\cJ}}{\cJ}\colon{}^{\tilde {\cJ}/}\cat\rightleftarrows{}^{{\cJ}/}\cat\colon(-)\circ p$$
compose to form the desired adjunction
\[(-)\star^P{\cJ}\colon\cat\rightleftarrows{}^{\tilde {\cJ}/}\cat\rightleftarrows{}^{{\cJ}/}\cat.\qedhere\]
\end{proof}

\begin{defn}
Given a diagram $d\colon {\cJ}\to {\cC}$, the \emph{category of weighted cones} in ${\cC}$ over the diagram $d\colon {\cJ}\to {\cC}$ is the weighted slice category ${\cC}^P_{/d}$.
In particular, the set of objects of ${\cC}^P_{/d}$ is given by
$$\Ob\left({\cC}^P_{/d}\right)\cong\Hom_{\cat}\left([0],{\cC}^P_{/d}\right)\cong\Hom_{{}^{{\cJ}/}\cat}([0]\star^P {\cJ},{\cC}),$$
and the set of morphisms of ${\cC}^P_{/d}$ could be unraveled from the expression
$$\Mor\left({\cC}^P_{/d}\right)\cong\Hom_{\cat}\left([1],{\cC}^P_{/d}\right)\cong\Hom_{{}^{{\cJ}/}\cat}([1]\star^P {\cJ},{\cC}).$$
\end{defn}

\begin{rmk}
Given a diagram $d\colon {\cJ}\to {\cC}$, and a weight $W\colon {\cJ}\to\set$, an object of the category of weighted cones ${\cC}^{el(W)}_{/D}$ is precisely a \emph{weighted cone} in the sense of \cite{RiehlCHT} and an \emph{indexing diagram} in the sense of \cite[\textsection3]{kelly}, i.e., a natural transformation $W\to\Hom_{{\cC}}(A,D-)$.
\end{rmk}

\begin{ex}
\label{example0}
Let $\Gamma$ denote the cospan shape category, which we display as
$$\xymatrix{a\ar[r]&b&\ar[l]c},$$
and let $\{0\}_{\Gamma}\colon\Gamma\to\set$ be the weight the constant weight
$$\xymatrix{\{0\}\ar[r]&\{0\}&\ar[l]\{0\}}.$$
Given any diagram $D\colon\Gamma\to\cC$ in a category $\cC$ with image
$$\xymatrix{D_{a}\ar[r]&D_{b}&\ar[l]D_{c}}.$$
a weighted cone  over $D$ consists of an object $C$ together with three morphisms $\lambda_a$, $\lambda_b$ and $\lambda_c$ that makes the following diagram commute:
$$\xymatrix{&C\ar[ld]_-{\lambda_a}\ar[d]|-{\lambda_b}\ar[rd]^-{\lambda_c}&\\
D_{a}\ar[r]&D_{b}&\ar[l]D_{c}.}$$
If then $W\colon\Gamma\to\set$ denotes the weight with image
$$\xymatrix{\{0\}\ar@{^{(}->}[r]&\{0,1\}&\ar@{_{(}->}[l]\{1\}},$$
a weighted cone over $D$ consists of an object $C$ together with four morphisms as displayed
$$\xymatrix{&C\ar[ld]_-{\lambda_a}\ar@/^/[d]^(.7){\lambda'_b}\ar@/_/[d]_(.7){\lambda_b}\ar[rd]^-{\lambda_c}&\\
D_{a}\ar[r]&D_{b}&\ar[l]D_{c}}.$$
subject to the condition that the two external triangles commute, but the two arrows $\lambda_b$ and $\lambda'_b$ do not necessarily agree. Note how the weight is encoding how many legs each piece of the diagram $D$ receives from the summit $C$.
\end{ex}

\subsection{The quasi-category of weighted cones}
\label{weightedconesqcat}
Aiming to define a quasi-category of cones over a diagram in a quasi-category, we extend the weighted join and slice constructions to simplicial sets, suitably generalizing Joyal's unweighted versions from \cite{joyalnotes}.

Recall from \cite[\textsection9.7]{joyalnotes} that the join of categories extends to a construction on simplicial sets, in a way that for every category ${\cI},{\cJ}$
$$N{\cI}\star N{\cJ}\cong N({\cI}\star {\cJ})$$
For every simplicial set $I$, the join with $J$ can be written levelwise as
$$(I\star J)_n\cong I_n\amalg J_n\amalg\left(\coprod_{i=0}^n(I_i\times J_{n-1-i})\right)\cong\coprod_{i=-1}^{n-1}(I_i\times J_{n+1-i}),$$
under the convention that $I_{-1}=J_{-1}=\varnothing$.

For any simplicial set $I$, there is a canonical inclusion $J\to I\star J$. In fact, If ${}^{J/}\sset$ denotes the slice category of $\sset$ over $J$, the \emph{join} with a simplicial set $J$ defines a functor
$$(-)\star J\colon\sset\to{}^{J/}\sset.$$
\label{sectionslice}
Recall from \cite[\textsection9.7]{joyalnotes} that the this functor admits a right adjoint
$$(-)_{/(-)}\colon{}^{J/}\sset\to\sset$$
that assigns to a diagram $d\colon J\to Q$ the \emph{slice} $K/_d$.
Moreover, the adjunction
$$(-)\star J\colon\sset\rightleftarrows{}^{J/}\sset\colon(-)_{/(-)}$$
is a Quillen pair between the Joyal model structure, and the corresponding sliced model structure category, as proven in \cite[Lemma 2.4.12]{RV1}.
In particular, any diagram $d\colon J\to Q$ in a quasi-category $Q$ is fibrant, and therefore the slice $Q_{/d}$ is a quasi-category.

\begin{defn}
The \emph{quasi-category of cones} in $Q$ over the diagram $d\colon J\to Q$ is the slice $Q_{/d}$.
The property of adjointness determines the slice construction up to isomorphism. Given a diagram $d\colon J\to Q$, the set of $n$-simplices of $Q_{/d}$ is given by the expression
$$\left(Q_{/d}\right)_n\cong\Hom_{\sset}(\Delta[n],Q_{/d})\cong\Hom_{{}^{J/}\sset}(\Delta[n]\star J,K),$$
which the willing reader could expand using the formula for the weighted join construction from \cref{weightedjoin}. In particular,
$$\left(Q_{/d}\right)_0\cong\Hom_{\sset}(\Delta[0],Q_{/d})\cong\Hom_{{}^{J/}\sset}(\Delta[0]\star J,K).$$
\end{defn}

We now define a weighted version of the join construction, whose corresponding slice construction models a simplicial set of cones weighted by any functor $p\colon\tilde J\to J$. The nature of the weight is clarified by the following remark, which explains why the homotopy theory $\kan$-valued functors is the same as that of left fibrations, which are the quasi-categorical analog of discrete fibrations of categories. $J$-valued left fibrations are the fibrant objects for a suitable model structure on $\sset_{/_J}$, called the \emph{covariant model structure}; we refer the reader to \cite[\textsection2.1]{htt} or \cite[\textsection2.3]{hm2} for a detailed account on this model structure.

\begin{rmk}
\label{straighteningunstraighteningqcats}
We record here the two instances of correspondances between simplicial functors and left fibrations in the quasi-categorical context.
\begin{itemize}[leftmargin=*]
	\item[(a)] Heuts and Moerdijk show in \cite[Theorem C]{hm} that for any category $\cJ$ there is a Quillen equivalence
	$$r_!\colon\sset_{/N\cJ}\rightleftarrows\Fun(\cJ,\sset)\colon r^*$$
between the covariant model structure and the projective model structure. This descends to an equivalence of homoopy theories
$$Left Fib_{/N\cJ}\simeq\Fun(\cJ,\kan)$$
between left fibrations $p\colon\tilde J\to N\cJ$ and $\kan$-valued functors $W\colon\cJ\to\kan$.

Given any functor $p\colon\tilde {J}\to N{\cJ}$, the \emph{rectification} $r_!(p)$ is obtained by taking a specific model of the homotopy fibers of $p$, i.e.,
$$r_!(p)(j):=\tilde J\times_{N\cJ} N(\cJ_{/_j})\simeq \tilde J\times^h_{N\cJ} N(\cJ_{/_j})\simeq\tilde J\times^h_{N\cJ}\{j\}=\hofib_j(p),$$
given that $N(\cJ_{/_j})\to N\cJ$ is a Kan-fibration and $N(\cJ_{/_j})\simeq\{j\}$.

Given any functor $W\colon \cJ\to\kan$, the construction $r^*(W)$ coincides with Lurie's relative nerve $N_W(\cJ)$ from \cite[\textsection 3.2.5]{htt}, as mentioned in in the introduction of \cite{hm}.
The property of adjointness determines the relative nerve construction up to isomorphism. The set of $n$-simplices of $r^*(W)$ over an $n$-simplex $\sigma$ of $NC$ is given by the expression
$$\begin{array}{rcl}
r^*(W)|_{\sigma}&\cong&\Hom((\sigma\colon\Delta[n]\to N\cJ),r^*(W\colon\cJ\to\sset))\\
&\cong&\Hom(r_!(\sigma\colon\Delta[n]\to N\cJ),(W\colon\cJ\to\sset)).
\end{array}$$
In particular, $r^*$ respects the terminal object, so the identity of $\cJ$ corresponds to the functor constant at the terminal category.
	\item[(b)] Lurie in \cite[Theorem 2.2.1.2]{htt} and Heuts-Moerdijk in \cite[Theorem C]{hm2} show that for any simplicial set $J$
	there is a Quillen equivalence
$$St\colon\sset_{/_J}\rightleftarrows s\Fun(\mathfrak C[J],\sset)\colon Un$$
between the covariant model structure on $\sset_{/_J}$ and the projective model structure on $s\Fun(\mathfrak C[J],\sset)$. Here, $\mathfrak C$ denotes the left adjoint to the homotopy coherent nerve functor $\mathfrak N$.
The Quillen equivalence descends to an equivalence of homotopy theories
$$Left Fib_{/_J}\simeq\sset\Fun(\mathfrak C[J],\kan)$$
between left fibrations $p\colon\tilde J\to J$ and simplicial functors $W\colon\mathfrak C[J]\to\kan$.

Given any functor $p\colon\tilde {J}\to N{\cJ}$, the straightening construction $St(p)$ is given by
$$\begin{array}{rcl}
St(p)(j)&:=&\Map_{\mathfrak C[\Delta[0]\star\tilde J]\amalg^{\mathfrak C[p]}_{\mathfrak C[\tilde J]}\mathfrak C[J]}(0,j)\\
&\cong&\Map_{\mathfrak C[(\Delta[0]\star\tilde J)\amalg^p_{\tilde J}J]}(0,j)\\
&=&\Map_{\mathfrak C[\Delta[0]\star^p J]}(0,j).
\end{array}$$
Given any functor $W\colon\mathfrak C[J]\to\sset$, the property of adjointness determines the unstraightening construction up to isomorphism. The set of $n$-simplices of $Un(W)$ over an $n$-simplex $\sigma$ of $X$ is given by the expression
$$\begin{array}{rcl}
Un(W)|_{\sigma}&\cong&\Hom(\sigma\colon\Delta[n]\to J,Un(W\colon\mathfrak C[J]\to\sset))\\
&\cong&\Hom(St(\sigma\colon\Delta[n]\to J),(W\colon\mathfrak C[J]\to\sset)).
\end{array}$$
In particular, $Un$ respects the terminal object, so the unstraightening of the identity of $J$ is the functor constant at the terminal category. If $J=N\cJ$ is the nerve of a category $\cJ$, the unstraightening construction coincides with the relative nerve.
\end{itemize}
\end{rmk}

We therefore focus on weights in the form of any simplicial map $p\colon\tilde J\to J$ (without requiring any further condition), and inspired by \cref{weightedjoinaspushout},
we define the weighted join as follows.

\begin{defn}
\label{weightedjoin} Let $p\colon\tilde J\to J$ be a simplicial map.
The \emph{join} of a simplicial set $I$ with $J$ \emph{weighted} by $p\colon\tilde J\to J$ is the simplicial set
$$I\star^pJ:=(I\star \tilde J)\amalg^p_{\tilde J}J.$$
It can be written levelwise as
$$(I\star^pJ)_n\cong I_n\amalg J_n\amalg\left(\coprod_{i=0}^n(I_i\times\tilde J_{n-1-i})\right).$$
For any simplicial set $I$, there is a canonical inclusion
$$J\to I\star\tilde J\to(I\star \tilde J)\amalg^p_{\tilde J}J=I\star^pJ.$$
In fact, the weighted join construction defines a functor
$$(-)\star^pJ\colon\sset\stackrel{-\star\tilde J}{\longrightarrow}{}^{\tilde J/}\sset\stackrel{-\amalg_{\tilde J}J}{\longrightarrow}{}^{J/}\sset.$$
\end{defn}

\begin{rmk}
There is a good compatibility between the nerve construction and the weighted join construction. Indeed, given categories $\cI$ and $\cJ$ and a discrete fibration  $P\colon\widetilde{\cJ}\to\cJ$, nerves of categories do not see the difference between $N(\cI\star^P\cJ)$ and $(N\cI)\star^{NP}(N\cJ)$;  indeed, for any category $\cC$ there is a canonical isomorphism
$$\Hom_{\sset}(N(\cI\star^P\cJ),N\cC)\cong\Hom_{\sset}((N\cI)\star^{NP}(N\cJ),N\cC).$$
We suspect that in fact nerve and weighted join construction commute with each other, but did not investigate this further.
\end{rmk}

We prove a weighted variant of Joyal's adjunction.

\begin{prop}
\label{adjunctionweightedjoinslice}
The join weighted by $p\colon\tilde J\to J$
admits a right adjoint given by
$$(-)^p_{/(-)}\colon{}^{J/}\sset\stackrel{(-)\circ p}{\longrightarrow}{}^{\tilde J/}\sset
\stackrel{(-)_{/(-)}}{\longrightarrow}
\sset.$$
Moreover, the adjunction
$$(-)\star^pJ\colon\sset\rightleftarrows{}^{\tilde J/}\sset\colon(-)^p_{/(-)}$$
is a Quillen pair between the Joyal model structure and the corresponding sliced model structure, in which fibrations, cofibrations and weak equivalences are created by forgetting down to $\sset$.
\end{prop}

\begin{proof}
The (dual of the) adjunction from \cite[Proposition 9.18]{awodey}
$$-\amalg_{\tilde J}J\colon{}^{\tilde J/}\sset\rightleftarrows{}^{J/}\sset\colon(-)\circ p,$$
is a Quillen pair. Indeed, the right adjoint does not change the underlying simplicial set, and therefore respects weak equivalences and fibrations.
Composing this Quillen pair with the Quillen pair from \cite[Lemma 2.4.12]{RV1}
$$(-)\star\tilde{J}\colon\sset\rightleftarrows{}^{\tilde J/}\sset\colon(-)_{/(-)}$$
we obtain the desired Quillen pair
\[(-)\star^pJ\colon\sset\rightleftarrows{}^{\tilde J/}\sset\rightleftarrows{}^{J/}\sset.\qedhere\]
\end{proof}

\begin{rmk}
Any diagram $d\colon J\to Q$ in a quasi-category $Q$ is fibrant, and therefore for any weight $p\colon\tilde J\to J$ the weighted slice $Q^p_{/d}$ is a quasi-category.
\end{rmk}

\begin{defn}
The \emph{quasi-category of weighted cones} in $Q$ over the diagram $d\colon J\to Q$ is the weighted slice $Q^p_{/d}$.
The property of adjointness determines the weighted slice construction up to isomorphism. Given a diagram $d\colon J\to Q$, the set of $n$-simplices of $Q^p_{/d}$ is given by the expression
$$\left(Q^p_{/d}\right)_n\cong\Hom_{\sset}(\Delta[n],Q^p_{/d})\cong\Hom_{{}^{J/}\sset}(\Delta[n]\star^p J,Q),$$
which the willing reader could expand using the formula for the weighted join construction from \cref{weightedjoin}. In particular,
$$\left(Q^p_{/d}\right)_0\cong\Hom_{\sset}(\Delta[0],Q^p_{/d})\cong\Hom_{{}^{J/}\sset}(\Delta[0]\star^p J,Q),$$
which allows us to think of the weighted slice $Q^p_{/d}$ as a simplicial set of cones in $Q$ over the diagram $d\colon J\to Q$ weighted by $p$.
\end{defn}

\begin{rmk}
The weighted slice construction generalizes at once Joyal's slice construction (to which it specializes when the weight is the identity of $J$), and the weighted slice construction for categories that we gave in \cref{weightedconescat} (to which it specializes when all the involved simplicial sets are nerves of categories and the weight is the nerve of a discrete fibration).
\end{rmk}

\begin{ex}
\label{example1}
Denote by $\tilde J$ the simplicial set obtained by gluing a standard $1$-simplex to a standard $2$-simplex as displayed.
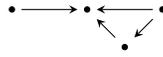
\[
\begin{tikzpicture}
[outer sep=0.3]
     \draw[fill] (1,0) circle (1pt) node(a1){};
     \draw[fill] (2,0) circle (1pt) node(a2){};
     \draw[fill] (3,0) circle (1pt) node(a3){};
     \draw[fill] (2.5,-0.5) circle (1pt) node(a4){};
     
     \draw[-stealth] (a1)--(a2);
     \draw[-stealth] (a3)--(a2);
     \draw[-stealth] (a4)--(a2);
     \draw[-stealth] (a3)--(a4);
     \end{tikzpicture}
     \]
Consider the weight
$$p\colon\xymatrix@C=2cm{\tilde J:=\Delta[1]\amalg_{\Delta[0]}\Delta[1]\ar[r]^-{\Delta[1]\amalg_{\Delta[0]}d}&\Delta[1]\amalg_{\Delta[0]}\Delta[2]\cong\Lambda^2[2]}$$
that collapsed the $2$-simplex to its first face by collapsing second face to a vertex.
This weight is the nerve of the Grothendieck construction of the functor $W\colon\Gamma\to\cat$ with image
$$\xymatrix{[0]\ar[r]^-1&[1]&\ar[l]_-{0}[0].}$$
Given any diagram $d\colon\Lambda^2[2]\to Q$ in a quasi-category $Q$ with image
$$\xymatrix{a\ar[r]^-f&b&\ar[l]_-gc.}$$
a weighted cone over $d$ consists of a vertex $\ell$ of $Q$ 
together with a $2$-simplex in $Q$ and a $3$-simplex in $Q$ glued along a $1$-simplex such that the $1$-skeleton looks as follows:
 
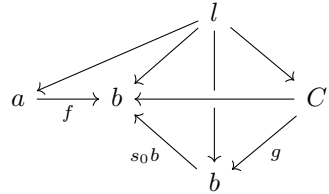
\[
     \begin{tikzcd}
     &&l\arrow[lld]\arrow[ld]\arrow[rd]\arrow[dd]&\\
     a \arrow[r, "f" swap]&b&&C\arrow[ld, "g"] \arrow[ll,crossing over]\\
     &&b\arrow[lu, "s_0b"]&
     \end{tikzcd}
     \]
Again, the weight is encoding how many $1$- and $2$-simplices each piece of the diagram $d$ receives from the summit $q$.
\end{ex}

\subsection{The quasi-category of fat weighted cones}
\label{weightedconesqcatfat}
In this subsection we generalize Joyal's fat join and fat slice constructions, providing a weighted version.
While they will be shown to be equivalent to their neat versions, we will see that the quasi-categories of weighted cones defined in terms of the fat join and slice makes sense in other models of $(\infty,1)$-categories rather than just quasi-categories.

Recall from \cite[\textsection9.17]{joyalnotes} that the \emph{fat join} with a simplicial set $J$ is the functor
$$(-)\diamond J\colon\sset\to{}^{J/}\sset$$
defined by the pushout
$$\xymatrix{(I\times J)\amalg(I\times J)\ar[r]\ar[d]&(I\times\Delta[1]\times J)\ar[d]\\
(I\amalg J)\ar[r]&I\diamond J}$$
of the folding map of $I\times J$ along the canonical map
$$(I\times J)\amalg(I\times J)\cong(I\times J)\times\partial\Delta[1]\to\Delta[1]\times I\times J\cong I\times\Delta[1]\times J.$$

Recall from \cite[\textsection9.17]{joyalnotes} that the fat join construction admits a right adjoint
$$(-)_{\fatslice (-)}\colon{}^{J/}\sset\to\sset$$
that assigns to a diagram $d\colon J\to Q$ the \emph{fat slice} $Q_{\fatslice d}$.
Moreover, the adjunction
$$(-)\diamond J\colon\sset\rightleftarrows{}^{\tilde J/}\sset\colon(-)_{\fatslice (-)}$$
is a Quillen pair between the Joyal model structure and the corresponding sliced model structure, as is proven in \cite[Lemma 2.4.12]{RV1}.

\begin{defn}
\label{weightedfatjoin}
The \emph{fat join} of a simplicial set $I$ with $J$ \emph{weighted} by $p\colon\tilde J\to J$ is the simplicial set
$$I\diamond^pJ:=(I\diamond \tilde J)\amalg^p_{\tilde J}J,$$
which comes with a canonical map $J\to I\diamond^pJ$.
The fat weighted join construction defines a functor
$$(-)\diamond^pJ\colon\sset\stackrel{-\diamond \tilde J}{\longrightarrow}{}^{\tilde J/}\sset\stackrel{-\amalg_{\tilde J}J}{\longrightarrow}{}^{J/}\sset.$$

\end{defn}

\begin{rmk}
\label{fatjoinwithDelta0}
The weighted fat join with $\Delta[0]$ can be understood as the pushout
$$\Delta[0]\diamond^pJ:=(\Delta[0]\diamond \tilde J)\amalg^p_{\tilde J}J\cong\left(\Delta[0]\amalg_{\Delta[1]}(\Delta[1]\times\tilde J)\right)\amalg^p_{\tilde J}J.$$
\end{rmk}

The following is a weighted variant of Joyal's adjunction, and can be proven similarly to \cref{adjunctionweightedjoinslice}

\begin{thm}
\label{adjunctionweightedfatjoinslice}
The join weighted by $p\colon\tilde J\to J$
admits a right adjoint
$$(-)^p_{\fatslice (-)}\colon{}^{J/}\sset\stackrel{(-)\circ p}{\longrightarrow}{}^{\tilde J/}\sset
\stackrel{(-)_{\fatslice (-)}}{\longrightarrow}
\sset.$$
Moreover, the adjunction
$$(-)\diamond^p J\colon\sset\rightleftarrows{}^{J/}\sset\colon(-)^p_{\fatslice (-)}$$
is a Quillen pair between the Joyal model structure and the corresponding sliced model structure.
\end{thm}

\begin{rmk}
Any diagram $d\colon J\to Q$ in a quasi-category $Q$ is fibrant, and therefore for any weight $p\colon\tilde J\to J$ the weighted slice $Q^p_{/d}$ is a quasi-category.
\end{rmk}

\begin{defn}
The \emph{quasi-category of fat weighted cones} in $Q$ over the diagram $d\colon J\to Q$ is the fat weighted slice $Q^p_{\fatslice d}$.
The property of adjointness determines the weighted fat slice construction up to isomorphism. Given a diagram $d\colon J\to Q$, the set of $n$-simplices of $Q^p_{\fatslice d}$ is given by the expression
$$\left(Q^p_{\fatslice d}\right)_n\cong\Hom_{\sset}(\Delta[n],Q^p_{\fatslice d})\cong\Hom_{{}^{J/}\sset}(\Delta[n]\star^p J,Q),$$
which the willing reader could expand using the formula for the weighted fat join construction from \cref{weightedfatjoin}. In particular,
$$\left(Q^p_{\fatslice d}\right)_0\cong\Hom_{\sset}(\Delta[0],Q^p_{\fatslice d})\cong\Hom_{{}^{J/}\sset}(\Delta[0]\star^p J,Q),$$
which allows us to think of the weighted fat slice $Q^p_{\fatslice d}$ as a simplicial set of cones in $Q$ over the diagram $d\colon J\to Q$ weighted by $p$.
\end{defn}

\begin{rmk}
Likewise its neat version, the weighted fat slice construction generalizes at once Joyal's fat slice construction and the weighted slice construction for categories that we gave in \cref{weightedconescat}.
\end{rmk}

\begin{ex}
Consider the weight
$$p\colon\Delta[1]\amalg_{\Delta[0]}\Delta[2]\to\Delta[1]\amalg_{\Delta[0]}\Delta[1]=\Lambda^2[2]$$
as in \cref{example1}.
Given any diagram $d\colon\Lambda^2[2]\to Q$ in a quasi-category $Q$ with image
$$\xymatrix{a\ar[r]^-f&b&\ar[l]_-gc}.$$
a weighted cone over $d$ consists of a vertex $\ell$ of $Q$ together with a cube $\square[2]:=\Delta[1]\times\Delta[1]$ in $Q$ and a prism $\Delta[2]\times \Delta[1]$ in $Q$ glued along copy of $\Delta[0]\times\Delta[1]$ subject to further conditions. For instance, the restriction to the top face of the square $\Delta[1]\times\Delta[0]$ and to the top face of the prism $\Delta[2]\times \Delta[0]$ is degenerate over $\ell$, and
the whole $1$-skeleton has looks as follows.
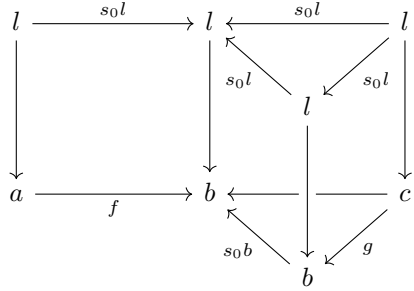
\[
     \begin{tikzcd}
     l\arrow[rr, "s_0l"]\arrow[dd]&&l\arrow[dd]&&l\arrow[dd]\arrow[ll, "s_0l" swap]\arrow[dl, "s_0l"]\\
     &\hphantom{l}&&l\arrow[ul, "s_0l"]&\\
     a\arrow[rr, "f" swap]&&b&&c\arrow[ll]\arrow[dl, "g"]\\
     &&&b\arrow[ul, "s_0b"]&
     %
     %arrow crossing over
     \arrow[from=2-4, to=4-4, crossing over]
     \end{tikzcd}
     \]
\end{ex}

As in \cite{RV2},
we denote by $\Delta_Q\downarrow_{Q^J}\id_{Q^J}$ or just $\Delta_Q\downarrow_{Q^J}Q^J$ the comma quasi-category\footnote{In general, the \emph{comma quasi-category} for any two given functors with the same target appears itself as weighted limit in the simplicial category $\qcat$, and is reviewed in \cref{comma}.} of the diagonal functor $\Delta_Q\colon Q\to Q^J$ and the identity of $Q^J$, 
and by $\Delta_{Q}\downarrow_{Q^J}d$ the comma quasi-category of the diagonal functor $\Delta_Q$ and the diagram $d$ seen as a simplicial map $d\colon\Delta[0]\to Q^J$. They are defined as the following pullbacks:
$$
\xymatrix{\Delta_Q\downarrow_{Q^J}Q^J\ar[d]_-{(p_0,p_1)}\ar[r]&Q^{J\times\Delta[1]}\ar@{->>}[d]^-{(p_0,p_1)}\\
Q\times Q^J\ar[r]_-{\Delta_Q\times Q^J}&Q^J\times Q^J}\quad\text{ and }\quad
\quad\xymatrix{\Delta_{Q}\downarrow_{Q^J}d\ar[d]_-{(p_0,p_1)}\ar[r]&Q^{J\times\Delta[1]}\ar@{->>}[d]^-{(p_0,p_1)}\\
Q\times\Delta[0]\ar[r]_{\Delta_Q\times d}&Q^J\times Q^J,}$$
where $p_0,p_1\colon Q^{J\times\Delta[1]}\cong(Q^J)^{\Delta[1]}\to Q^{J}$ denote the domain and codomain projections, respectively.
By \cite[Lemma 4.2.3]{RVbook}, there is an isomorphism of quasi-categories
$$Q^{\Delta[0]\diamond J}\cong\Delta_Q\downarrow_{Q^J}Q^J\cong (Q\times Q^J)\times_{Q\times Q}Q^{\Delta[1]}.$$
Pulling back this isomorphism, the fat slice of a diagram $d\colon J\to Q$ can be identified with the comma quasi-category of the diagonal functor and $d$,
$$Q_{\fatslice d}\cong\Delta_{Q}\downarrow_{Q^J}d\cong (Q\times\Delta[0])\times_{Q\times Q}Q^{\Delta[1]}.$$

We now prove a weighted version of this fact.

\begin{prop}
\label{fatslicevscomma}
Let $d\colon J\to Q$ be a diagram in $Q$.
There is an isomorphism
\begin{equation}
\label{equationfatslice}
Q^p_{\fatslice d}\cong\Delta_{Q}\downarrow_{Q^{\tilde J}}(d\circ p)=Q_{\fatslice _{}d\circ p}.
\end{equation}
\end{prop}

\begin{proof}
Using the expression from \cref{fatjoinwithDelta0} of $\Delta[0]\diamond^p J$ as a pushout, we obtain an isomorphism of quasi-categories
$$Q^{\Delta[0]\diamond^p J}\cong Q\times^{\Delta}_{Q^{\tilde J}}Q^{\Delta[1]\times\tilde J}\times^{p^*}_{Q^{\tilde J}} Q^J.$$
By further pulling back along $d\colon\Delta[0]\to Q^J$, we obtain the isomorphism
\[Q^p_{\fatslice d}\cong Q\times^{\Delta}_{Q^{\tilde J}}Q^{\Delta[1]\times\tilde J}\times^{d\circ p}_{Q^{\tilde J}}\Delta[0]\cong\Delta_{Q}\downarrow_{Q^{\tilde J}}(d\circ p)=Q_{\fatslice _{}d\circ p}.\qedhere\]
\end{proof}

\begin{digression}
The formula for the fat slice established in the proposition above makes sense in any category endowed with simplicial cotensors and certain pullbacks.
One can therefore make sense of the $\infty$-category of weighted cones over any diagram $d\colon J\to Q$ in any $(\infty,2)$-category presented by a suitably complete $\qcat$-enriched category (e.g.~the category of complete Segal spaces with the enrichment discussed in \cite[Example 2.2.5]{RV4}) by means of the formula (\ref{equationfatslice}).
\end{digression}

\subsection{Comparison of models}
\label{Comparison of models}
The aim of this subsection is to show that the neat slice and the fat slice over a diagram $d\colon J\to Q$ are equivalent.

Recall from \cite[\textsection9.18]{joyalnotes} that for any simplicial set $I$ there is a categorical equivalence
$$I\diamond J\stackrel{\simeq}{\longrightarrow}I\star J$$
between the fat join and the neat join, which is natural in $I$. This is also true in the weighted case.

\begin{prop}
\label{equivalencefatneatjoin}
For any simplicial set $I$ there is a categorical equivalence
$$I\diamond^pJ\stackrel{\simeq}{\longrightarrow}I\star^pJ,$$
which is natural in $I$.
\end{prop}

\begin{proof}
By \cite[Proposition A.2.4.4]{htt}
pushouts where one leg is a cofibration are homotopy pushouts. We conclude using the definitions of the two weighted join constructions and the fact that the map $I\diamond\tilde J\to I\star\tilde J$ between the unweighted joins with $\tilde J$ is a categorical equivalence.
\end{proof}

Recall from \cite[\textsection9.18]{joyalnotes} that, given a diagram $d\colon J\to Q$ in $Q$, regarded as a map $d\colon\Delta[0]\to Q^J$,
there is a categorical equivalence
$$Q_{/d}\stackrel{\simeq}{\longrightarrow}Q_{\fatslice d}$$
between the slice of $d$ and the fat slice of $d$. This is also true in the weighted case.

\begin{prop}
\label{slicevsfatslice}
Given a diagram $d\colon J\to Q$ in $Q$ regarded as a map $d\colon\Delta[0]\to Q^J$,
there is an equivalence of quasi-categories
$$Q^p_{/d}\stackrel{\simeq}{\longrightarrow}Q^p_{\fatslice d}$$
between the weighted slice of $d$ and the weighted fat slice of $d$.
\end{prop}

The proof is an enhancement of the argument from \cite[Proposition 2.4.13]{RV1}.

\begin{proof}
We consider the Quillen
adjunctions established in \cref{adjunctionweightedjoinslice,adjunctionweightedfatjoinslice}.
By \cref{equivalencefatneatjoin}, the canonical map
$I\diamond^pJ\stackrel{\simeq}{\longrightarrow}I\star^pJ$
between the fat join and the neat join is a weak equivalence. By \cite[Corollary 1.4.4(b)]{hovey}, this is equivalent to saying that the canonical map between the corresponding
right Quillen functors
$Q^p_{/d}\stackrel{\simeq}{\longrightarrow}Q^p_{\fatslice d}$
is a weak equivalence whenever $d\colon J\to Q$ is fibrant in ${}^{J/}\sset$, i.e., whenever $Q$ is a quasi-category.
\end{proof}

We conclude the subsection
by conjecturing alternative expressions for the quasi-category of weighted cones over a diagram $d\colon J\to Q$. We will see at the end of \cref{Weighted limits a quasi-category}
how these candidate descriptions
could be used to show that our definition of weighted limit of the diagram $d\colon J\to Q$ is compatible with the approach by Gepner-Haugseng-Nikolaus from \cite[\textsection2]{GHN}. However, the reader does not need to read this part to understand the rest of the paper, and should feel free to continue to \cref{weightedlimits}.

For any quasi-category $Q$, there is a functor
$$\Hom_Q(-,-)\colon Q^{\op}\times Q\to\mathfrak N[\kan],$$
which can be understood as the straightening of the canonical fibration
$$(\partial_0,\partial_1)\colon\Tw(J)\to J^{\op}\times J,$$
where $\Tw(J)$ denotes the \emph{twisted arrow quasi-category}. For details on these constructions, see e.g.~ in \cite[\textsection2]{BGN}.

We also recall that, given any weight $p\colon\tilde J\to J$, we can consider its straightening $St(p)\colon\mathfrak C[J]\to\sset$, as explained in \cref{straighteningunstraighteningqcats}(2), and then take its adjoint $st(p)\colon J\to\mathfrak N[\sset]$. Moreover, we can take its bifibrant replacement in the projective model structure, which is a functor $st(p)^{bifib}\colon J\to\mathfrak N[\kan]$. Inspired by what happens with ordinary categories, we expect the following to be true, which gives a third equivalent model for the quasi-category of weighted cones.

\begin{conjecture}
\label{comparisonGHN}
Let $Q$ be a quasi-category, $d\colon J\to Q$ a diagram and $p\colon\tilde J\to J$ a weight.
There is a Joyal equivalence
$$Q^p_{\fatslice _d}\simeq st(p)^{bifib}\downarrow_{\mathfrak N[\kan]^J}\Hom_{Q}(-,d),$$
where $st(p)^{bifib}\downarrow_{\mathfrak N[\kan]^J}\Hom_{Q}(-,d)$ denotes the comma quasi-category of the functors $st(p)^{bifib}\colon\Delta[0]\to\mathfrak N[\kan]^J$ and $\Hom_{Q}(-,d)\colon Q\to\mathfrak N[\kan]^J$.
\end{conjecture}
\label{defcotensored}
In presence of a notion of two-variable adjunction of quasi-categories, a quasi-category $Q$ is said to be
\emph{tensored and cotensored} over $\mathfrak N[\kan]$ if there exist functors
$$-\otimes-\colon\mathfrak N[\kan]\times Q\to Q\text{ and }[-,-]\colon\mathfrak N[\kan]^{\op}\times Q\to Q$$
such that, together with $\hom_Q$ they form a two-variable adjunction of quasi-categories
$$(-\otimes-,[-,-],\Hom_{Q}(-,-))\colon \mathfrak N[\kan]\times Q\to Q.$$
This should be the case, for instance, when $Q=\mathfrak N[\cQ]$ is the homotopy coherent nerve of a $\kan$-category that is tensored and cotensored over $\kan$.
The following is likely to be a formal property of two-variable adjunctions of quasi-categories, and would allow to further expand the formula for $Q^p_{\fatslice _d}$.

\begin{conjecture}
\label{sliceforcotensored}
Let $Q$ be a
quasi-category that is tensored and cotensored over $\mathfrak N[\kan]$, $d\colon J\to Q$ a diagram and $p\colon\tilde J\to J$ a weight.
There is a Joyal equivalence
$$st(p)^{bifib}\downarrow_{\mathfrak N[\kan]^J}\Hom_{Q}(-,d)\simeq(st(p)^{bifib}\otimes-)\downarrow_{Q^J}d\simeq Q\downarrow_Q[[st(p)^{bifib},d]],$$
where $[[st(p)^{bifib},d]]$ is the object of $Q$ defined as
$$[[st(p)^{bifib},d]]:=\lim\big(\xymatrix@C=1.1cm{\Tw(J)\ar[r]^-{(\partial_0,\partial_1)}&J^{\op}\times J\ar[rr]^-{(st(d)^{bifib})^{\op}\times d}&&\mathfrak N[\kan]^{\op}\times Q\ar[r]^-{[-,-]}&Q}\big).$$
\end{conjecture}

\section{Weighted limits}
\label{Weighted limits}

\label{weightedlimits}
Aiming to understand homotopy weighted limits in $(\infty,1)$-category theory, we start by describing in \cref{Weighted limits in unenriched category theory} weghted limits in unenriched category theory, namely in ordinary categories. Given a category $\cC$, the limit of a diagram $D\colon\cJ\to\cC$ weighted by a functor $W\colon\cJ\to\set$ enjoys two equivalent universal properties: one in terms of homsets and one in terms of its category of cones.

We then address the question of what should be the weighted limit of a diagram in an $(\infty,1)$-category $\cQ$, as opposed to an ordinary $1$-category $\cC$.
By taking this step, we are adding two levels of complexity, as we are introducing a higher categorical structure, and we want to work with derived or homotopically meaningful constructions.

To this end, we expect to replace isomorphisms by equivalences, homsets by derived mapping spaces, categories of cones by $(\infty,1)$-categories of cones, terminal objects in a category to terminal objects in an $(\infty,1)$-category, weights valued in sets with weights valued in spaces (namely Kan complexes).
Depending on the employed model of $(\infty,1)$-category, one of the two universal properties discussed in the unenriched case generalizes more naturally.

When working with an $(\infty,1)$-category in the form of a $\kan$-category $\cQ$, (strict) limits for diagrams $D\colon\cJ\to\cQ$ weighted by diagrams $W\colon\cJ\to\kan$ are just a special instance of the theory of weighted limits in enriched category theory. We revise this in \cref{Weighted limits in a Kan-category}. Then, in \cref{Homotopy weighted limits in a Kan-enriched category} we take the homotopical point of view and discuss homotopy weighted limits in the $\kan$-category $\cQ$ with respect to the homotopical structure determined by the $\kan$-enrichment. We explain how this approach, which requires minimal assumptions on $\cQ$, compares to others existing in the literature, for which $\cQ$ is assumed to be either the category of bifibrant objects of a simplicial model structure or the $(\infty,1)$-core of an $\infty$-cosmos.

When working with an $(\infty,1)$-category in the form of an arbitrary quasi-category $Q$,
we take over from the definition of ordinary (namely unweighted) limits for diagrams $d\colon J\to Q$, which is introduced by Joyal \cite{joyalnotes} as the terminal object in the quasi-category of cones. In \cref{Weighted limits a quasi-category}, we generalize suitably these ideas for any given weight in the form of a simplicial map $p\colon\tilde J\to J$, and define weighted limits in a quasi-category using the constructions of the quasi-category of weighted cones from \cref{The (quasi-)category of weighted cones}. We show in \cref{weightedlimitcotensored} that our definition of weighted limit agrees with the approach by Gepner-Haugseng-Nikolaus from \cite[\textsection2]{GHN} when $Q$ is tensored and cotensored over the quasi-category of spaces.

Finally, in \cref{Comparison of weighted limits in different models} we compare the two approaches (for $\kan$-categories and quasi-categories) and show that they agree, validating our constructions.

\subsection{Weighted limits in unenriched category theory}
\label{Weighted limits in unenriched category theory}

Let $\cC$ be an ordinary category.
Given a diagram $D\colon\cJ\to\cC$ in $\cC$ and a weight $W\colon\cJ\to\set$, a weighted cone over $D$ consists of an object $A$ and of a natural transformation
 $$\Psi\colon W\to\Map_{\cC}(A ,D-)$$
 and we constructed in \cref{weightedconescat} the category $\cC^{el_{\cJ}(W)}_{/D}$ of weighted cones over the diagram $D$, with respect to the weight $el_{\cJ}\colon\cJ\to\set$.
 
 We now give two (equivalent) universal properties for weighted limits in the unenriched context. One naturally generalizes to the enriched context (see \cref{Weighted limits in a Kan-category,Homotopy weighted limits in a Kan-enriched category}), and the other one to the quasi-categorical context (see \cref{Weighted limits a quasi-category}).

\begin{defn}
\label{weightedlimitsuniversalpropertydiscrete}
The \emph{limit} of a diagram $D\colon \cJ\to\cC$ \emph{weighted} by a weight $W\colon \cJ\to\set$ consists of an object $L$ of $\cC$ together with
a weighted cone
$$\Lambda\colon W\to\Map_{\cC}(L,D)$$
such that one of the two following equivalent conditions holds.
\begin{enumerate}
\item[(a)] The cone $\Lambda$ induces a natural bijection
$$\Hom_{\cC}(A,L)\cong\Hom_{\Fun(\cJ,\set)}(W,\Hom_{\cC}(A,D-)).$$
\item[(b)] The cone $\Lambda$ is terminal in the category $\cC^{el(W)}_{/D}$ of weighted cones over the diagram $D$.
\end{enumerate}
\end{defn}

By the Yoneda lemma the limit of $D$ weighted by $W$ is determined up to isomorphism.

\begin{rmk}
The weighted limit exists if and only if there exists an object $\lim{}^WD$ for which a natural bijection like the one in (b) holds. In that case, the universal cone can be retrieved as image of the identity. So it makes sense to just talk about the weighted limit, rather than the full cone. When the weighted limit exists, we write $L\cong\lim{}^WD$.
\end{rmk}

\begin{rmk}
When the weight is the constant functor $\Delta[0]_{\cC}\colon\cJ\to\set$, the weighted limit of a diagram $D\colon\cJ\to\cC$ is the ordinary limit, i.e.,
$$\lim{}^{\Delta[0]_{\cJ}}D\cong\lim D.$$
\end{rmk}

\begin{rmk}
\label{existencestrictweightedlimits}
If $\cC$ is a complete category, given any weight $W\colon \cJ\to\set$ and any diagram $D\colon \cJ\to\cC$, the limit of $D$ weighted by $W$ and is given by the formula
$$\lim{}^WD\cong\int_{j\in\cJ}\prod_{Wj}Dj$$
and also by the formula
$$\lim{}^WD\cong\lim\big(\xymatrix@C=1.3cm{\Tw{\cJ}\ar[r]^-{(\partial_0,\partial_1)}&\cJ^{\op}\times\cJ\ar[r]^-{W^{\op}\times D}&\set\times\cQ\ar[r]^-{-\otimes-}&\cQ}\big),$$
where $\Tw\cJ$ denotes the twisted arrow category of $\cJ$; cf.~\cite[Remark 1.13]{loregian}.
The two equivalent constructions define a functor
$$\lim{}^{W}\colon\Fun(\cJ,\cC)\to\cC.$$
\end{rmk}

\begin{ex}
Let $\Delta[0]_{\Gamma}\colon\Gamma\to\set$ be the constant weight
$$\xymatrix{\{0\}\ar[r]&\{0\}&\ar[l]\{0\}}.$$
Given any diagram $D\colon\Gamma\to\cC$ in a category $\cC$ with image
$$\xymatrix{A\ar[r]^-{F}&B&\ar[l]_-{G}C},$$
the weighted limit cone is the pullback $A\times_B C$ of $F$ and $G$, together with the universal cone as depicted
$$\xymatrix@R=1cm{&A\times_B C\ar[ld]_-{pr_1}\ar[d]|(.6){F\circ pr_1=G\circ pr_2}\ar[rd]^-{pr_2}&\\
A\ar[r]_-{F}&B&\ar[l]^-{G}C.}$$
If instead $W\colon\Gamma\to\set$ is the weight
$$\xymatrix{\{0\}\ar@{^{(}->}[r]&\{0,1\}&\ar@{_{(}->}[l]\{1\},}$$
the weighted limit cone is the product $A\times C$ together with the universal cone as depicted
$$\xymatrix{&A\times C\ar[ld]_-{pr_1}\ar@/^/[d]^(.7){G\circ pr_2}\ar@/_/[d]_(.7){F\circ pr_1}\ar[rd]^-{pr_2}&\\
A\ar[r]_-{F}&B&\ar[l]^-{G}C.}$$
\end{ex}

In unenriched category theory, all limits weighted by $W$ can be realized as ordinary limits over the category of elements of $W$.

\begin{prop}[{\cite[Example 7.1.5]{RiehlCHT}}]
\label{weightedlimitsdiscretearelimits}
Let $W\colon \cJ\to\set$ be a weight. The weighted limit of a diagram $D\colon\cJ\to\cC$ can be realized as
$$\lim{}^WD\cong\lim\left(el_{\cJ}(W)\to \cJ\stackrel{D}{\longrightarrow}\cC\right).$$
\end{prop}

We will show as \cref{allweightedhomototopylimitsareconstantenriched,allweightedhomototopylimitsareconical}
analogs of this fact for weighted limits in a $\kan$-category (with additional structure) and in a quasi-category, respectively.

%%%
%%%
%%%

\subsection{Weighted limits in a simplicial category}
\label{Weighted limits in a Kan-category}

Let $\cQ$ be a $\sset$-enriched category, or for short a simplicial category. This means that $\cQ$ consists of a set of objects $\Ob(\cQ)$, together with suitably compatible mapping simplicial sets $\Map_{\cQ}(A,B)$ for any objects $A$ and $B$. In particular, the mapping object defines a functor
$$\Map_{\cQ}(-,-)\colon\cQ^{\op}\times\cQ\to\sset.$$
Given two simplicial categories\footnote{To avoid size issues, all simplicial categories are assumed to be small unless otherwise specified.} $\cQ$ and $\cQ'$, simplicial functors between them, namely functors that are compatible with the simplicial enrichment, assemble into a category $s\Fun(\cQ,\cQ')$ whose morphisms are given by simplicial natural transformations. We refer the reader to \cite[\textsection 2.2]{kelly} for a more detailed construction on this account.

Given any simplicial category $\cJ$, a diagram $D\colon\cJ\to\cQ$ in $\cQ$ and a weight $W\colon\cJ\to\sset$ (in the form of simplicial functors), a \emph{weighted cone} over $D$ consists of an object $A$ and of a simplicial natural transformation
 $$\Psi\colon W\to\Map_{\cQ}(A ,D-).$$
The weighted limit of $D$ is then defined as the universal cone for the given weight $W$.
 
\begin{defn}
\label{weightedlimitsuniversalproperty}
The \emph{limit} of a diagram $D\colon \cJ\to\cQ$ \emph{weighted} by a weight $W\colon \cJ\to\sset$
consists of an object $L$ of $\cQ$ together with
a weighted cone
$$\Lambda\colon W\to\Map_{\cQ}(L,D)$$
such that the cone $\Lambda$ induces a natural bijection
$$\Map_{\cQ}(A,L)\cong\Map_{s\Fun(\cJ,\kan)}(W,\Map_{\cQ}(A,D-)).$$
\end{defn}

By the Yoneda lemma the limit of $D$ weighted by $W$ is determined up to isomorphism.

\begin{rmk}
The weighted limit exists if and only if there exists an object $\lim{}^WD$ for which a natural isomorphism like the above holds. In that case, the universal cone can be retrieved as image of the identity. So it makes sense to just talk about the weighted limit, rather than the full cone. We write $L\simeq\lim{}^WD$.
\end{rmk}

\begin{rmk}
When the weight is the constant functor $\Delta[0]_{\cC}\colon\cJ\to\sset$, the weighted limit of a diagram $D\colon\cJ\to\cC$ is sometimes called \emph{conical} or \emph{enriched} limit, given that weighted cones coincide with ordinary cones. We write
$$\lim{}^{\Delta[0]_{\cJ}}D=:\lim D.$$
\end{rmk}

\begin{rmk}
\label{endformula}
If $\cQ$ is complete as a $1$-category as well as cotensored over $\sset$, it is shown in \cite[\textsection 3.10]{kelly} that given any weight $W\colon \cJ\to\sset$ and any diagram $D\colon \cJ\to\cQ$, the limit of $D$ weighted by $W$ exists, and is given by the enriched end$$\lim{}^WD\cong\int_{j\in\cJ}Dj^{Wj}.$$
Moreover, the resulting construction defines a bifunctor
$$\lim{}^{-}(-)\colon\Fun(\cJ,\sset)^{\op}\times\Fun(\cJ,\cQ)\to\cQ.$$
\end{rmk}

\begin{ex}
\label{comma}
Define the \emph{comma weight} as the functor $W\colon\Gamma\to\sset$ whose image is
$$\xymatrix{\{0\}\ar@{^{(}->}[r]&\Delta[1]&\ar@{_{(}->}[l]\{1\}}.$$
Given any diagram $D\colon\Gamma\to\cQ$ in a simplicial category $\cQ$ with image
$$\xymatrix{A\ar[r]^-{F}&B&\ar[l]_-{G}C},$$
the weighted limit cone of $D$ is called the \emph{comma construction} $f\downarrow_{B} g$.
This weighted limit appears e.g.~ in \cite[\textsection3.3]{RV2} when $\cQ$ is the simplicial category of quasi-categories (or in fact any $\infty$-cosmos), and in \cite[\textsection4]{kelly2limits} when $\cQ$ is the $2$-category of categories.
Under suitable completeness assumptions on $\cQ$ we can specialize and unravel the end formula for $\lim{}^WD$ from \cref{existencestrictweightedlimits},
and obtain the expression
$$F\downarrow_{B} G\cong\eq\left(A\times B^{\Delta[1]}\times C\rightrightarrows B\times B\right)\cong (A\times C)\times_{B\times B}B^{\Delta[1]},$$
namely the comma is the pullback
$$\xymatrix{F\downarrow_{B}G\ar[d]_-{(p_0,p_1)}\ar[r]&B^{\Delta[1]}\ar@{->>}[d]^-{(p_0,p_1)}\\
A\times C\ar[r]_{F\times G}&B\times B.}$$
It comes with the universal cone
$$\xymatrix@R=2cm{&F\downarrow_{B} G\ar@{}[d]|-{\Rightarrow}\ar[ld]_-{p_0}\ar@/^/[d]^(.7){G\circ p_1}\ar@/_/[d]_(.7){F\circ p_0}\ar[rd]^-{p_1}&\\
A\ar[r]_-{F}&B&\ar[l]^-{G}C.}$$
This cone is usually depicted as
$$\xymatrix@R=.5cm{F\downarrow_{B} G\ar@{}[rd]|-{\Rightarrow}\ar[r]^-{p_1}\ar[d]_-{p_0}&C\ar[d]^-G\\
A\ar[r]_-F&B.}$$
\end{ex}

Comparing with \cref{weightedlimitsdiscretearelimits}, it is not true that any weighted limit of a diagram $D\colon\cJ\to\cQ$ can be realized as the ordinary limit of a diagram obtained by precomposing $D$ with a suitable functor $\widetilde{\cJ}\to\cJ$. We give a counterexample for $\cJ=[0]$ and general $\cQ$.

\begin{rmk}
Given any object $D$ of $\cQ$, seen as diagram $D\colon[0]\to\cQ$, its limit weghted by $\Delta[1]\colon[0]\to\sset$ is by definition the cotensor
$$\lim{}^{\Delta[1]}D\cong D^{\Delta[1]}.$$
However, for any category $\tilde{\cJ}$ we have that
$$\lim\left(\tilde{\cJ}\to[0]\stackrel{D}{\longrightarrow}\cQ\right)\cong D,$$
and in general the cotensor $D^{\Delta[1]}$ is different from $D$.
\end{rmk}

We will see in \cref{allweightedhomototopylimitsareconstantenriched} that the property holds if we move to a homotopical context and consider homotopically meaningful weighted limits, as we discuss in the coming section.

\subsection{Homotopy weighted limits in a $\kan$-enriched category}
\label{Homotopy weighted limits in a Kan-enriched category}

From now on, we want to think that the simplicial category $\cQ$ is an $(\infty,1)$-category, and for this we specialize to the case of $\cQ$ being in fact a $\kan$-category. This means that 
any $\Map_{\cQ}(A,B)$ is now a Kan complex, and the mapping object defines a functor
$$\Map_{\cQ}(-,-)\times\cQ^{\op}\times\cQ\to\kan.$$

While the strict notions of weighted limits in $\cQ$ from the previous section are well-defined, they are not homotopically meaningful. To make the notion of weighted limit homotopically meaningful, one way is to modify the defining universal property in a way that all occurring mapping simplicial sets are derived, and require a condition of the form
$$\Map^h_{\cQ}(A,\lim{}^WD )\simeq\Map^h_{\Fun(\cJ,\sset)}(W,\Map^h_{\cQ}(A,D-)).$$

To this end, we first make $\cQ$ is also a homotopical category, when endowed with the following weak equivalences.

\begin{defn}
\label{weakequivalences}
A map $f\colon X\to Y$ in $\cQ$ is a \emph{weak equivalence} if for any object $Z$
the induced maps
$$f^*\colon\Map_{\cQ}(Y,Z)\to\Map_{\cQ}(X,Z)\text{ and }f_*\colon\Map_{\cQ}(Z,X)\to\Map_{\cQ}(Z,Y)$$
are weak equivalences of Kan complexes.
\end{defn}

The following is a straightforward consequence of the definition of a homotopical category.

\begin{prop}
The category $\cQ$ with the weak equivalences above is a homotopical category, i.e., it satisfies the two-out-of-six property.
\end{prop}

\begin{digression}
The structure required on $\cQ$, namely the enrichment and the compatible homotopical structure, is a reminiscence of certain structures considered in the literature to present an $(\infty,1)$-category.
More precisely, one could take as $\cQ$:
\begin{enumerate}[leftmargin=*]
	\item the category $\cM_{fib}$ of fibrant objects in a simplicial model category $\cM$ with all objects cofibrant, with the given enrichement and the weak equivalences inherited from the model structure; or more generally
	\item the $(\infty,1)$-core $core_*\cK$ of an $\infty$-cosmos $\cK$, with the given enrichement and the weak equivalences inherited from the weak equivalences of $\cK$.
\end{enumerate}

An $\infty$-cosmos\footnote{We assume the definition of $\infty$-cosmos that Riehl-Verity use in \cite{RV7}, which is stronger than the one in firstly introduced in \cite{RV2}. In particular, every object is cofibrant.}  $\cK$ consists of a (possibly large) category enriched over $\qcat$ with a specified class of maps called ``isofibrations'' satisfying some axioms. For instance, the class of fibrant object in a model structure enriched over the Joyal model structure with all objects cofibrant forms an $\infty$-cosmos, in which the isofibrations are precisely the fibrations between fibrant objects.
In an $\infty$-cosmos, the quasi-categorical enrichment and the class of isofibrations determine a notion of weak equivalence, given by the maps $f\colon X\to Y$ in $\cK$ that induce Joyal equivalences on hom-quasi-categories
$$f^*\colon\Map_{\cK}(Y,Z)\to\Map_{\cK}(X,Z)\text{ and }f_*\colon\Map_{\cK}(Z,X)\to\Map_{\cK}(Z,Y)$$
for any object $Z$ in $\cK$. As mentioned in \cite[Remark 2.1.3]{RV7}, the $\infty$-cosmos underlying a model structure enriched over the Joyal model structure with all objects cofibrant, these coincide with the weak equivalences between fibrant objects in the model structure.

Given an $\infty$-cosmos $\cK$, the $(\infty,1)$-core of $\cK$ is a $\kan$-category $core_*\cK$ with the same set of objects, and hom Kan complexes given by
$$\Map_{core_*\cK}(X,Y):=core(\Map_{\cK}(X,Y)),$$
where $core$ denotes the maximal Kan complex of a quasi-category.
Given a simplicial model structure (rather than one just enriched over the Joyal model structure) with all objects cofibrant, the $(\infty,1)$-core of its underlying $\infty$-cosmos coincides with the $\infty$-cosmos itself, i.e., it is given by its category of fibrant objects with the given enrichment, which happens to be over $\kan$. In particular, (1) is a special case of (2). It will be shown in \cite[\textsection16]{RVbook} that a map $f\colon X\to Y$ in the $(\infty,1)$-core $core_*\cK$ of an $\infty$-cosmos $\cK$ is a weak equivalence of $\cK$ if and only if it induces weak equivalences of Kan complexes
$$f^*\colon\Map_{core_*\cK}(Y,Z)\to\Map_{core_*\cK}(X,Z)\text{ and }f_*\colon\Map_{core_*\cK}(Z,X)\to\Map_{core_*\cK}(Z,Y)$$
for any object $Z$. In particular, the homotopical structure induced by $\cK$ to $core_*\cK$ is detected by the $\kan$-enrichment of the core, as requested by \cref{weakequivalences}.
\end{digression}

The leading example for $\cQ$ is the category $\kan$ of Kan complexes with the usual enrichment. Other interesting examples, such as the $\kan$-category of quasi-categories or the $\kan$-category of complete Segal spaces, can be realized as the $(\infty,1)$-cores of the corresponding $\infty$-cosmoi.

Thanks to the structure of a homotopical category on $\cQ$, it makes sense to construct its homotopy category $\mathcal Ho(\cQ)$, say when two objects are equivalent, and define left and right derived functors whose source or target involves $\cQ$. Indeed, we recall that it makes sense to define the derived functor of any functor between homotopical categories, as its best approximation to a homotopical functor. We refer the reader to \cite[Chapter 2]{RiehlCHT} for a survey on the general theory of derived functors via deformations, first due to Shulman \cite{shulman}.

Aiming to make sense of the three derived mapping spaces
$$\Map^h_{\cQ}(A,\lim{}^WD ),\ \Map^h_{\cQ}(A,D-)\ \text{ and }\ \Map^h_{s\Fun(\cJ,\sset)}(W,\Map^h_{\cQ}(A,D-)),$$
we now proceed to constructing the functors
$$\Map^h_{\cQ}(-,-)\colon\cQ^{\op}\times\cQ\to\kan\quad\text{ and}$$
$$\Map^h_{s\Fun(\cJ,\sset)}(-,-)\colon(s\Fun(\cJ,\sset))^{\op}\times s\Fun(\cJ,\sset)\to\kan,$$
as derived functor of the mapping space functors of $\cQ$ and $s\Fun(\cJ,\sset)$.

The following proposition is an easy consequence of the definition of weak equivalences in $\cQ$, and in particular says that the mapping spaces in $\cQ$ have already the correct homotopy types.

\begin{prop}
The functor $\Map_{\cQ}(-,-)\colon\cQ^{\op}\times\cQ\to\kan$
is homotopical.
\end{prop}

It follows that the derived mapping space in $\cQ$ can be realized as the strict mapping space
$$\Map_{\cQ}^h(A,B)\simeq\Map_{\cQ}(A,B).$$

To make sense of the derived mapping spaces in $\Fun(\cJ,\sset)$, we need to endow also the category $s\Fun(\cJ,\sset)$ with a homotopical structure. For this, we evoke an instance of \cite[Theorem 5.4]{moser} and consider in fact a model structure.
When unspecified, we assume the category $\sset$ to be endowed with the Kan-Quillen model structure.

\begin{prop}
\label{projectivemodelstructure}
For any small simplicial category $\cJ$, the projective model structure exists on the category $s\Fun(\cJ,\sset)$ of $\cJ$-shaped diagrams in $\sset$ and makes it into a simplicial model category.
\end{prop}

In this model structure, the fibrant diagrams are precisely the the diagrams $W\colon\cJ\to \kan\subset\sset$ that are valued in Kan complexes, such as the functor $\Map^h_{\cQ}(A,D-)$ for any $D\colon\cJ\to\cQ$.
Not every weight is projectively cofibrant. We will give explicit models of cofibrant replacements for fibrant weights in \cref{cofibrantreplacement}.

Having endowed the simplicial category $s\Fun(\cJ,\sset)$ with a model structure, and in particular a homotopical structure, derived mapping spaces make sense.

By \cite[Lemma 2.5 (ii)]{gambino}
the functor
$$\Map_{\Fun(\cJ,\sset)}(-,-)\colon\Fun(\cJ,\sset)^{\op}\times\Fun(\cJ,\sset)\to\kan$$
is a Quillen bifunctor, 
so the derived mapping space can be realized as
$$\Map^h_{\Fun(\cJ,\sset)}(W,W')\simeq\Map_{\Fun(\cJ,\sset)}(W^{cof},W')$$
whenever $W'$ is fibrant.

Now that we understand both the expressions
$$\Map^h_{\cQ}(A,\lim{}^WD )\ \text{ and }\ \Map^h_{\Fun(\cJ,\sset)}(W,\Map^h_{\cQ}(A,D-)),$$
we are ready to give the universal property of weighted homotopy limits.
\begin{defn}
\label{weightedlimitsenrichedcategories}
The \emph{homotopy limit} of a diagram $D\colon\cJ\to\cQ$ \emph{weighted} by $W\colon\cJ\to\sset$ consists of an object $L$ of $\cQ$
together with
a simplicial natural transformation
$$\Lambda\colon W\to\Map^h_{\cQ}(L,D)$$
that induces a (simplicially) natural Joyal equivalence
$$\Map_{\cQ}^h(A,L)\simeq\Map^h_{\Fun(\cJ,\sset)}(W,\Map^h_{\cQ}(A,D-)).$$
\end{defn}

By the Yoneda lemma the limit of $D$ weighted by $W$ is determined up to equivalence.

\begin{rmk}
The homotopy weighted limit exists if and only if there exists an object $L$ for which a natural equivalence like the one above holds. In that case, the universal cone can be retrieved as image of the identity. So it makes sense to just talk about the homotopy weighted limit, rather than the full cone. We write $L\simeq\holim^WD$.
\end{rmk}

\begin{rmk}
When the weight is the constant functor $\Delta[0]_{\cQ}\colon\cJ\to\sset$, the homotopy weighted limit of a diagram $D\colon\cJ\to\cQ$ is the ordinary homotopy limit, i.e.,
$$\holim{}^{\Delta[0]_{\cJ}}D\simeq \holim D.$$
\end{rmk}

This notion of homotopy weighted limit agrees with the one given in \cite{RV7}, given when the weight $W$ is ``flexible'', and with the one used in \cite{gambino} when $\cQ$ is the category of bifibrant objects of a simplicial model structure.

Rather than deriving the universal property, one could instead try to derive the construction of strict weighted limit, aiming to find a meaningful object. The following proposition shows that the two approaches agree, as well as provides an explicit construction for homotopy weighted limits.

\begin{prop}
\label{formulahomotopylimit}
Let $D\colon\cJ\to\cQ$ be a diagram and $\cJ\to\sset$ be a weight. If $\widetilde W\colon\cJ\to\sset$ is a projective cofibrant replacement of $W$, and the strict limit of $D$ weighted by $\widetilde W$ exists in $\cQ$, then it also computes the homotopy limit if $D$ weighted by $W$, namely$$\lim{}^{\widetilde W}D\simeq\holim^{W}D.\footnote{The reader should be aware that this expression contains an abuse of notation, as the left hand side is define up to isomorphism and the right hand side is only defined up to equivalence. It should be read as: the equivalence class of the strict weighted limit represents the homotopy weighted limit.}
$$
\end{prop}

\begin{proof}
We observe that there are equivalences
$$\begin{array}{rcl}
\Map_{\cQ}^h(A,\lim{}^{\widetilde W}D )&\simeq&\Map_{\cQ}(A,\lim{}^{\widetilde W}D )\\
&\cong&\Map_{\Fun(\cJ,\sset)}(\widetilde W,\Map_{\cQ}(A,D-))\\
&\cong&\Map^h_{\Fun(\cJ,\sset)}(\widetilde W,\Map_{\cQ}(A,D-))\\
&\simeq&\Map^h_{\Fun(\cJ,\sset)}(W,\Map^h_{\cQ}(A,D-)),
\end{array}$$
which show that $\lim^{\widetilde W}D$ has the desired universal property.
\end{proof}

Recall from \cref{existencestrictweightedlimits}
 that if $\cQ$ is complete and cotensored over $\sset$, then the strict weighted limit construction defines a functor
$$\lim{}^{(-)}(-)\colon s\Fun(\cJ,\sset)^{\op}\times s\Fun(\cJ,\cQ)\to\cQ.$$
We can endow the category $s\Fun(\cJ,\cQ)$ with the homotopical structure induced levelwise by that of $\cQ$. The analogous homotopical structure on $s\Fun(\cJ,\sset)$ is precisely the homotopical structure underlying the projective model structure from \cref{projectivemodelstructure}.

\begin{thm}
\label{weightedhomotopylimitsarederivedfunctors}
Let $\cQ$ be complete and cotensored over $\sset$.
The homotopy weighted limit exists, it is unique up to equivalence,
and can be computed as the left derived functor of the weighted limit functor
$$\lim{}^{(-)}(-)\colon s\Fun(\cJ,\sset)^{\op}\times s\Fun(\cJ,\cQ)\to\cQ.$$
\end{thm}

The argument involves the formula for left derived functors in terms of left deformations. We refer the reader to \cite[\textsection 2.2]{RiehlCHT} for more details on this account.

\begin{proof}[Proof of \cref{weightedhomotopylimitsarederivedfunctors}]
Choose a functorial projectively cofibrant replacement
$$(-)^{cof}\colon s\Fun(\cJ,\sset)\to s\Fun(\cJ,\sset).$$
By the completeness assumption of $\cQ$, we know that for any weight $W\colon\cJ\to\sset$ and $D\colon\cJ\to\cQ$ the strict weighted limit $\lim{}^{W^{cof}}D$ exists in $\cQ$. In particular, by \cref{formulahomotopylimit} we know that
$$\lim{}^{W^{cof}}D\simeq\holim{}^{W}D.$$
We now show that $\lim{}^{W^{cof}}D$ also computes the left derived functor $L\lim^D$ of the strict weighted limit construction, so that
$$\lim{}^{W^{cof}}D\simeq R\lim{}^{W}D.$$

We observe that any weak equivalence $W_1\to W_2$ between projectively cofibrant weights induces equivalences
$$\begin{array}{rcl}
\Map_{\cQ}^h(A,\lim{}^{W_2}D )&\simeq&\Map^h_{\cQ}(A,\lim{}^{W_2}D )\\
&\simeq&\Map^h_{\cQ}(A,\lim{}^{W_1}D )\\
&\simeq&\Map^h_{\Fun(\cJ,\sset)}(W_1,\Map^h_{\cQ}(A,D-)),
\end{array}$$
so that in particular $\lim{}^{W_1}D\simeq \lim{}^{W_2}D$.
This means that the weighted limit functor
$$\lim{}^{(-)}(-)\colon (s\Fun(\cJ,\sset)_{cof})^{\op}\times s\Fun(\cJ,\cQ)\to\cQ$$
is homotopical, if $s\Fun(\cJ,\sset)_{cof}$ denotes the category of projectively cofibrant weights.
In other words, the functor
$(W,D)\mapsto (W^{cof},D)$
provides a \emph{left deformation} (see \cite[Def. 2.2.1]{RiehlCHT}) for the functor
\[\lim{}^{(-)}(-)\colon s\Fun(\cJ,\sset)^{\op}\times s\Fun(\cJ,\cQ)\to\cQ.$$
with respect to the category $s\Fun(\cJ,\sset)_{cof})^{\op}\times s\Fun(\cJ,\cQ)$, if we denote by $\Fun(\cJ,\sset)_{cof}$  the full subcategory of cofibrant objects.
By \cite[Theorem 2.2.8]{RiehlCHT} we can then compute the left derived functor as
$$R\lim{}^{W}D\simeq\lim{}^{W^{cof}}D.\qedhere\]
\end{proof}

It follows that the weighted homotopy limit defined a functor
$$\lim{}^{-}-\colon(s\Fun(\cJ,\sset))^{\op}\times s\Fun(\cJ,\cQ)\to\cQ,$$
which can be realized as
$$\holim^WD\simeq\lim{}^{W^{cof}}D$$
or, if the weight $W$ is already cofibrant, as
$$\holim^WD\simeq\lim{}^WD.$$
In the case of $\cQ$ being the category of bifibrant objects of a simplicial model structure $\cM$, this recovers a result of Gambino \cite{gambino}.

This formulas puts emphasis on cofibrant weights.
The adjunctions from \cref{straighteningunstraighteningqcats} can be used to produce canonical cofibrant replacements of fibrant weights $W\colon\cJ\to\kan$ whenever $\cJ$ is an (unenriched) $1$-category or the homotopy coherent realization of a simplicial set $J$.

\begin{prop}
\label{cofibrantreplacement}
\begin{enumerate}[leftmargin=*]
	\item For any $1$-category $\cJ$ and any weight $W\colon\cJ\to\kan$, the diagram
$$W^{cof}:=r_!\circ r^*(W)\colon\cJ\to\sset$$
is a cofibrant replacement for $W$.
	\item For any simplicial set $J$, and any weight $W\colon\mathfrak C[J]\to\kan$, the diagram
	$$W^{cof}:=St\circ Un(W)\colon\mathfrak C[J]\to\sset$$
is a cofibrant replacement for $W$.
\end{enumerate}
\end{prop}

\begin{proof}
Since the relative nerve $r^*(W)$ (as well as any other object) is cofibrant in the covariant model structure, the derived counit of the diagram $W$ can be realized as the strict counit $\epsilon_W\colon (r_!\circ r^*)(W)\to W$.
Given that the pair $(r_!,r^*)$ is a Quillen equivalence and the diagram $W$ is fibrant by assumption, the (derived) counit $\epsilon_W$ is a weak equivalence. In particular $W$ is equivalent to $W^{cof}=(r_!\circ r^*)(W)$, which is cofibrant since $r_!$ is left Quillen and $r^*(W)$ is cofibrant. This concludes the proof of (1), and the proof of (2) is analogous.
\end{proof}

It is interesting to note that both models of cofibrant replacements specialize to established models of cofibrant replacements for the constant weights at $\Delta[0]$.

\begin{rmk}
\label{replacementconical weight}
The value of the cofibrant replacement of the constant weight $\Delta[0]\colon N\cJ\to\sset$ can be computed using the formulas from \cref{straighteningunstraighteningqcats}.
\begin{enumerate}[leftmargin=*]
	\item[(a)] When $\cJ$ is an ordinary $1$-category, the value of the cofibrant replacement of the constant weight $\Delta[0]_{N\cJ}\colon N\cJ\to\kan$ on an object $j$ of $\cJ$ can be computed as
$$\begin{array}{rclr}
\quad\quad\quad(r_!\circ r^*(\Delta[0]_{N\cJ}))(j)&\cong&(r_!(\id_{N\cJ}: N\cJ\to N\cJ))(j)\\
&\cong&N\cJ\times_{N\cJ}N(\cJ_{/_j})\\
&\cong&N(\cJ_{/_j}).
\end{array}$$
%$$\begin{array}{rclr}
%\quad\quad\quad(r_!\circ r^*(\Delta[0]_{N\cJ}))(j)&\cong&(r_!(\id_{N\cJ}: N\cJ\to N\cJ))(j)\\
%&\cong&\mathrm{colim}_{\sigma:\Delta[n]\to N\cJ}(r_!(\sigma:\Delta[n]\to N\cJ))(j)\\
%&\cong&\mathrm{colim}_{\sigma:\Delta[n]\to N\cJ}(\Delta[n]\times^{\sigma}_{N\cJ}N\cJ_{/_j})\\
%&\cong&\mathrm{colim}_{\sigma:\Delta[n]\to N\cJ}N([n]\times^{\sigma}_{\cJ}\cJ_{/_j})\\
%&\cong&\mathrm{colim}_{\sigma:\Delta[n]\to N\cJ}N(\sigma\downarrow j)&\\
%&\cong&\mathrm{colim}_{\sigma:\Delta[n]\to N(\cJ_{/_j})}\Delta[n]&\\
%&\cong&N(\cJ_{/_j}).&
%\end{array}$$
This is precisely the weight for homotopy limits which was considered e.g.~in \cite[\textsection 14.8.5]{Hirschhorn} or \cite[Example 5.2.8]{RV2} and identified as a cofibrant replacement for the constant weight $\Delta[0]_{N\cJ}\colon\cJ\to\kan$.
	\item[(b)] When $\cJ=\mathfrak C[J]$ is the homotopy coherent realization of a simplicial set $J$, the value of the cofibrant replacement of the constant weight $\Delta[0]_{\mathfrak C[J]}\colon\mathfrak C[J]\to\kan$ on a vertex $j$ of $J$ can be computed as
$$\begin{array}{rcl}
(St\circ Un(\Delta[0]_{\mathfrak C[J]}))(j)&\cong&(St(\id_J: J\to J))(j)\\
&\cong&\Map_{\mathfrak C[(\Delta[0]\star J)\amalg_JJ]}(0,j)\\
&\cong&\Map_{\mathfrak C[\Delta[0]\star J]}(0,j).
\end{array}$$
This is precisely the weight for homotopy limits considerd in \cite[Definition 4.1.6]{RV7} and identified as a cofibrant replacement for the constant weight $\Delta[0]_{\mathfrak C[J]}\colon\mathfrak C[J]\to\kan$.
\end{enumerate}
\end{rmk}

\begin{ex}
Recall that $\Gamma$ denotes the cospan shape category $a\leftarrow b\rightarrow c$.
Consider the constant weight $\Delta[0]_{\Gamma}\colon\Gamma\to\sset$
$$\xymatrix{\Delta[0]\ar[r]&\Delta[0]&\ar[l]\Delta[0].}$$
Given any diagram $\Gamma\to\cQ$ in a $\kan$-category $\cQ$,
$$\xymatrix{A\ar[r]^-F&B&\ar[l]_-GC,}$$
the corresponding homotopy weighted limit should be naturally though of as a homotopy pullback of $F$ and $G$, so we write
$$\holim^{\Delta[0]}D=:A\times^h_BC.$$
Given that this weight is not projectively cofibrant, we discuss several cofibrant replacements, and the corresponding model of homotopy pullback.
\begin{enumerate}[leftmargin=*]
	\item Both the approaches from \cref{replacementconical weight} gives the cofibrant replacement
$\Delta[0]_{\Gamma}^{cof}$ given by
$$\xymatrix{\Delta[0]\ar@{^{(}->}[r]&\Lambda^2[2]&\ar@{_{(}->}[l]\Delta[0]}.$$
Indeed, with the formulas from (a) we obtain
$$\Delta[0]_{\Gamma}^{cof}(a)=N(\Gamma_{/_a})\cong\Delta[0]\cong N(\Gamma_{/_c})=\Delta[0]_{\Gamma}^{cof}(c)$$
$$\quad\text{and }\Delta[0]_{\Gamma}^{cof}(b)=N(\Gamma_{/_b})\cong\Delta[1]\amalg_{\Delta[0]}\Delta[1]=\Lambda^2[2].$$
On the other hand, noticing that $\Gamma\cong\mathfrak C[N\Gamma]$, the cofibrant replacement (b) is also available. Using for instance the description of mapping simplicial sets of $\mathfrak C[\Delta[0]\star\Gamma$ in terms of \emph{necklaces} in $\Delta[0]\star\Gamma$ (see \cref{appendix}), we see that
$$
\quad\quad\Delta[0]_{\Gamma}^{cof}(a)=\Map_{\mathfrak C[\Delta[0]\star\Gamma]}(0,a)\cong\Delta[0]\cong\Map_{\mathfrak C[\Delta[0]\star\Gamma]}(0,a)\cong\Delta[0]_{\Gamma}^{cof}(c),
$$
$$
\text{ and }\Delta[0]_{\Gamma}^{cof}(b)=\Map_{\mathfrak C[\Delta[0]\star\Gamma]}(0,b)\cong\Lambda^2[2].
$$
In particular, the homotopy pullback of any functor $D\colon\cJ\to\cQ$ in a $\kan$-category $\cQ$ can be computed as
$$A\times^h_{B} C:=\holim^{\Delta[0]_{\Gamma}} D\simeq\lim{}^{\Delta[0]_{\Gamma}^{cof}}D.$$
	\item The comma weight $W^{comma}\colon\Gamma\to\sset$ from \cref{comma},
$$\xymatrix{\{0\}\ar@{^{(}->}[r]&\Delta[1]&\ar@{_{(}->}[l]\{1\}},$$
is also a projectively cofibrant replacement for the constant weight.
In particular, the homotopy pullback of any functor $D\colon\cJ\to\cQ$ in a $\kan$-category $\cQ$ can be computed as
$$A\times^h_{B} C:=\holim^{\Delta[0]} D\simeq\lim{}^{W^{comma}}D=F\downarrow_{B}C.$$
$$A\times^h_{B} C:=\holim^{\Delta[0]_{\Gamma}} D\simeq\lim{}^{\Delta[0]_{\Gamma}^{cof}}D.$$
\end{enumerate}
\end{ex}

Using upcoming results we will be able to prove at the end of \cref{Comparison of weighted limits in different models} the following analog of \cref{weightedlimitsdiscretearelimits}.

\begin{thm}
\label{allweightedhomototopylimitsareconstantenriched}
Let $W\colon\mathfrak C[J]\to\kan$ be a weight and $\cQ$ a $\kan$-category that admits all weighted homotopy limits. The homotopy weighted limit of a diagram $D\colon\mathfrak C[J]\to\cQ$ can be realized as
$$\holim{}^WD\simeq\holim\left(\mathfrak C[Un(W)]\to\mathfrak C[J]\stackrel{D}{\longrightarrow}\cQ\right).$$
\end{thm}

%%%%
%%%%
%%%%
%%%%

\subsection{Weighted limits in a quasi-category}
\label{Weighted limits a quasi-category}
When working with $(\infty,1)$-cat\-egories in the form of quasi-categories, we are able to generalize the universal property for weighted limits in the form (b) from \cref{weightedlimitsuniversalpropertydiscrete}, thanks to the construction of the quasi-category of weighted cones, discussed in \cref{weightedconesqcat,weightedconesqcatfat}.

Let $Q$ be a an $(\infty,1)$-category in the form of a quasi-category. A cone over the diagram $d\colon J\to Q$ consists of a vertex $a$ of $Q$ together with a limit cone
$\lambda\colon\Delta[0]\to Q_{/d}$ from $a$ to $d$. By definition of slice, $\lambda$ is represented by a map $\lambda\colon\Delta[0]\star^p J\to Q$ that makes the following diagram commute
$$\xymatrix@C=3cm@R=.5cm{J\ar[rd]^-d\ar@{^{(}->}[d]&\\
\Delta[0]\star^p J\ar[r]|-\lambda&Q.\\
\Delta[0]\ar@{_{(}->}[u]\ar[ru]_-l&}$$

We recall the notion of (unweighted)
limits for diagrams $d\colon J\to Q$, which is treated in several sources, such as \cite{joyalquasicategories,htt,RV1}.

\begin{defn}[{\cite[Definition 4.1]{joyalquasicategories}}]
The terminal object of $Q$ is a vertex $t$ that enjoys the following lifting property for $n\ge1$:
$$\xymatrix@R=.3cm{\Delta[0]\ar[r]_-n\ar@/^2pc/[rr]^-{t}&\partial\Delta[n]\ar[r]\ar@{^{(}->}[d]&Q.\\
&\Delta[n]\ar@{-->}@/_1pc/[ru]&}$$
\end{defn}

\begin{defn}[{\cite[Definition 4.5]{joyalquasicategories}}]
The \emph{limit cone} of a diagram $d\colon J\to Q$ consists of a vertex $\ell$ together with a limit cone
$\lambda\colon\Delta[0]\star J\to Q$
from $\lim d$ to $d$
that is teminal in the quasi-category of cones $Q_{/d}$.
\end{defn}

By definition of terminal objects, the condition means that $\lambda$ enjoys the following lifting property for $n\ge1$
$$\xymatrix@R=.3cm{\Delta[0]\ar[r]_-n\ar@/^2pc/[rr]^-{\lambda}&\partial\Delta[n]\ar[r]\ar@{^{(}->}[d]&Q_{/d}.\\
&\Delta[n]\ar@{-->}@/_1pc/[ru]&}$$
Using the join-slice adjunction, the lifting problem can be rewritten as
$$\xymatrix@R=.3cm{\Delta[0]\star J\ar[r]_-n\ar@/^2pc/[rr]^-{\lambda}&\partial\Delta[n]\star J\ar[r]\ar@{^{(}->}[d]&Q.\\
&\Delta[n]\star J\ar@{-->}@/_1pc/[ru]&}$$

\begin{rmk}
\label{uniquenesslimit}
By \cite[Proposition 3.5.3]{RVscratch}, there is a categorical equivalence
$$Q\downarrow_Q\ell\simeq\Delta_Q\downarrow_{Q^J}d\cong Q_{\fatslice d}\simeq Q_{/d},$$
which says that the limit $\ell$ represents the quasi-category of cones over $d$.
The quasi-categorical Yoneda Lemma
implies that the limit $\ell$ is determined up to equivalence in $Q$. When the limit exists, we write $\ell\simeq\lim d$.

In fact, this condition is sufficient: the diagram $d\colon J\to Q$ admits a limit in $Q$ if and only if the quasi-category of cones is representable, i.e., if there exists an object $\ell$ in $Q$ such that
$$Q\downarrow_Q\ell\simeq Q_{/d},$$
and in that case the universal cone can be retrieved as the image of the identity of $\ell$. So it makes sense to just talk about the limit, rather than the full cone.

By \cite[Proposition 5.2.11]{RV2}, if $Q$ is a complete quasi-category, namely it admits all limits of shape $J$, the limit defines a functor
$$\lim\colon Q^J\to Q.$$
\end{rmk}

We now proceed to defining the limit of $J$-shaped diagrams in a quasi-category weighted by some weight $p\colon\tilde J\to J$. We recall that the nature of this weight was previously discussed in \cref{straighteningunstraighteningqcats}.

\begin{defn}
\label{weightedlimitquasi-category}
The \emph{limit} of a diagram $d\colon J\to Q$ \emph{weighted} by a weight $p\colon\tilde J\to J$ is a vertex $\ell$ of $Q$  together with a weighted cone
$$\lambda\colon\Delta[0]\star^p J\to Q$$
from $\ell$ to $d$ that is teminal in the quasi-category of weighted cones $Q^{p}_{/d}$.
\end{defn}

By definition of terminal object, the condition means that $\lambda$ enjoys the following lifting property for $n\ge1$:
$$\xymatrix@R=.3cm{\Delta[0]\ar[r]_-n\ar@/^2pc/[rr]^-{\lambda}&\partial\Delta[n]\ar[r]\ar@{^{(}->}[d]&Q^{p}_{/d}.\\
&\Delta[n]\ar@{-->}@/_1pc/[ru]&}$$
Using the weighted join-slice adjunction from \cref{adjunctionweightedjoinslice}, the lifting problem can be rewritten as
$$\xymatrix@R=.3cm{\Delta[0]\star^p J\ar[r]_-n\ar@/^2pc/[rr]^-{\lambda}&\partial\Delta[n]\star^p J\ar[r]\ar@{^{(}->}[d]&Q.\\
&\Delta[n]\star^p J\ar@{-->}@/_1pc/[ru]&}$$

\begin{rmk}
\label{uniquenessweightedlimit}
Combining \cite[Proposition 3.5.3]{RVscratch} with \cref{slicevsfatslice,fatslicevscomma} we find is a categorical equivalence
$$Q\downarrow_Q\ell\simeq\Delta_Q\downarrow_{Q^{\tilde J}}(d\circ p)\cong Q^{p}_{\fatslice d}\simeq Q^{p}_{/d},$$
which says that the weighted limit $\ell$ represents the quasi-category of cones over the diagram $d$.
By the quasi-categorical Yoneda Lemma implies that the weighted limit $\ell$ is determined up to equivalence in $Q$. When the weighted limit exists, we write $\ell\simeq\lim{}^pd$.

As before, diagram $d\colon J\to Q$ admits a limit weighted by $p\colon\tilde J\to J$ in $Q$ if and only if the quasi-category of weighted cones is representable, i.e., if there exists an object $\ell$ in $Q$ such that
$$Q\downarrow_Q\ell\simeq Q^p_{/d},$$
and in that case the universal weighted cone can be retrieved as the image of the identity of $\ell$. So it makes sense to just talk about the weighted limit, rather than the full cone.
\end{rmk}

\begin{rmk}
When the weight is the identity $\id_{J}\colon J\to J$, the weighted limit of a diagram $d\colon J\to Q$ is the ordinary limit, i.e.,
$$\lim{}^{\id_J}d\simeq\lim d.$$
\end{rmk}

The following shows that every weighted limit in a quasi-category can be realized as an ordinary limit of a suitably fatter diagram, and generalizes \cref{weightedlimitsdiscretearelimits}.
\begin{thm}
\label{allweightedhomototopylimitsareconical}
Let $p\colon\tilde J\to J$ be a weight. The homotopy weighted limit of a diagram $d\colon J\to Q$ can be realized as
$$\lim{}^pd\simeq\lim\left(d\circ p\colon\tilde J\stackrel{p}{\longrightarrow}J\stackrel{d}{\longrightarrow}Q\right).$$
\end{thm}

The fact that every weighted limit can be computed as an unweighted limit of an appropriate modified diagram is a well-known property for ordinary limits in $1$-categories, and it is not surprising that it holds here, since the universal properties of (weighted) limits in quasi-categories keep into account automatically the correct homotopical behavior of all objects involved.

\begin{proof}
Consider a weight $p\colon\tilde J\to J$ and a diagram $d\colon J\to Q$.
By \cref{uniquenesslimit}, the limit of the diagram $d\circ p$ has the following universal property
$$Q\downarrow_Q\lim{}(d\circ p)\simeq\Delta_Q\downarrow_{Q^{\tilde J}}d\circ p.$$
As a consequence of \cref{uniquenessweightedlimit}, it follows that \[\lim{}^pd\simeq\lim{}(d\circ p).\qedhere\]
\end{proof}

\begin{rmk}
As a consequence of \cref{allweightedhomototopylimitsareconical}, if $Q$ is a complete quasi-category, namely it admits all limits, it also admits all the weighted limits.
\end{rmk}

Assuming the conjectured models of quasi-categories of weighted cones from \cref{comparisonGHN,sliceforcotensored}, we can easily deduce that, when $Q$ is tensored and cotensored over $\mathfrak N[\kan]$, the definition of weighted limit recovers that by Gepner-Haugseng-Nikolaus from \cite[\textsection2]{GHN}.

\begin{consequence}
\label{weightedlimitcotensored}
If $Q$ is a quasi-category that is tensored and cotensored over $\mathfrak N[\kan]$, for any diagram $d\colon J\to Q$ and any weight $p\colon\tilde J\to J$ the weighted limit of $d$ can be computed as
$$\lim{}^pd\simeq\lim\big(\xymatrix@C=1.2cm{\Tw(J)\ar[r]^-{(\partial_0,\partial_1)}&J^{\op}\times J\ar[rr]^-{(st(p)^{bifib})^{\op}\times d}&&\mathfrak N[\kan]^{op}\times Q\ar[r]^-{[-,-]}&Q}\big).$$
\end{consequence}

\begin{proof}[Conditional proof]
Assuming \cref{comparisonGHN,sliceforcotensored}, we would have Joyal equivalences
$$Q^p_{\fatslice _d}\simeq st(p)^{bifib}\downarrow_{\mathfrak N[\kan]^J}\Hom_{Q}(-,d)\simeq Q\downarrow_Q[[st(p)^{bifib},d]].$$
It would then follow that
\begin{equation}
\label{equationend}
\lim{}^pd\simeq[[st(p)^{bifib},d]]:=\lim\big(\Tw(J)\to J^{\op}\times J\to\mathfrak N[\kan]^{\op}\times Q\to Q\big),
\end{equation}
as desired.
\end{proof}

\subsection{Comparison of weighted limits in different models}
\label{Comparison of weighted limits in different models}

Denote by $\mathfrak C\colon\scat\to\sset$
the \emph{homotopy coherent realization},
namely the left adjoint of the Cordier's \emph{homotopy coherent nerve functor} $\mathfrak N\colon\sset\to\scat$ from \cite{cordier}.  The descriptions of these functors are reviewed e.g.~in \cite{riehlnecklace}.

In virtue of this adjunction, any  diagram $D\colon\mathfrak C[J]\to\cQ$ in a simplicial category $\cQ$, which we think of as a \emph{$J$-shaped homotopy coherent},
naturally transpose to a $J$-shaped diagram $d\colon J\to\mathfrak N[\cQ]$
in its homotopy coherent nerve $\mathfrak N[\cQ]$. On the other hand, any homotopy coherent weight $W\colon\mathfrak C[J]\to\sset$ can be \emph{unstraightened} to a map $Un(W)\colon\tilde J\to J$, as discussed in \cref{straighteningunstraighteningqcats}(b).

Assume as in \cref{Homotopy weighted limits in a Kan-enriched category} that $\cQ$ is a $\kan$-enriched category together with the intrinsec notion of weak equivalences determined by the enrichment.
The main result shows that the theory of homotopy weighted limits described in \cref{Homotopy weighted limits in a Kan-enriched category} for $\cQ$ and the theory of weighted limits developed in \cref{Weighted limits a quasi-category} for its homotopy coherent nerve $\mathfrak N[\cQ]$ agree.

Recall the cofibrant replacement construction $W^{cof}$ for weights $W\colon\mathfrak C[J]\to\kan$ from \cref{cofibrantreplacement}(b).

\begin{thm}
\label{maintheorem}
Let $J$ be a small category. For any weight $W\colon\mathfrak C[J]\to\kan$ and $D\colon\mathfrak C[J]\to\cQ$ a homotopy coherent diagram in $\cQ$, if $\holim^{W}D$ exists then $\lim{}^{Un(W)}d$ exists and
$$\lim{}^{Un(W)}d\simeq\holim^{W}D.$$
\end{thm}

\begin{rmk}
We observe that, whenever $\cQ$ admits cotensors and pullbacks, the theorem
given an explicit model for the limit of $d\colon J\to\mathfrak N[\cQ]$ weighted by $Un(W)$ . Indeed, using the analysis on weighted homotopy limits discussed in \cref{Homotopy weighted limits in a Kan-enriched category},
we can express the weighted limit of $d$ as a strict weighted limit as follows:
$$\begin{array}{rcl}
\lim{}^{Un(W)}d&\simeq&\holim{}^{W^{cof}}D\\
&\simeq&\lim{}^{W^{cof}}D\\
&\cong& \int_{j\in\mathfrak C[J]}D(j)^{W^{cof}(j)}=\int_{j\in\mathfrak C[J]}D(j)^{\Map_{\mathfrak C[\Delta[0]\star^{Un(W)}J]}(0,j)}.
\end{array}$$
\end{rmk}

To prove \cref{maintheorem}, it will be crucial to understand the homotopy coherent realization $\mathfrak C[I\star^pJ]$ of a weighted join $I\star^pJ$.
While the set of objects of $\mathfrak C[I\star^pJ]$ can be easily described as
$$\Ob(\mathfrak C[I\star^pJ])=(I\star^pJ)_0=I_0\amalg J_0,$$the mapping spaces require some work.
The following theorem, whose unweighted version first appeared as \cite[Theorem 5.3.19]{RV7}, describes the mapping spaces of $\mathfrak C[I\star^pJ]$ in terms of the mapping spaces of the simplicial categories $\mathfrak C[I\star\Delta[0]]$, $\mathfrak C[\Delta[0]\star^pJ]$, $\mathfrak C[I]$ and $\mathfrak C[J]$.
The proof exploits the description of the simplices of any homotopy coherent realization as necklaces, and is postponed until the appendix.

\begin{thm}
\label{mainfactweighted}
Let $p\colon\tilde J\to J$ be a simplicial map, and $I$ a simplicial set.
\begin{enumerate}[leftmargin=*]
\item For any vertices $a$ and $b$ of $I$ and $J$ respectively, there is an isomorphism of simplicial sets
$$\Map_{\mathfrak C[I\star^p J]}(a,b)\cong\Map_{\mathfrak C[I\star\Delta[0]]}(a,\top)\times\Map_{\mathfrak C[\Delta[0]\star^p J]}(\bot,b).$$
\item For any vertices $a$ and $a'$ of $I$, there is an isomorphism of simplicial sets
$$\Map_{\mathfrak C[I\star^p J]}(a,a')\cong\Map_{\mathfrak C[I]}(a,a').$$
\item For any vertices $b$ and $b'$ of $J$, there is an isomorphism of simplicial sets
$$\Map_{\mathfrak C[I\star^p J]}(b,b')\cong\Map_{\mathfrak C[J]}(b,b').$$
\item For any vertices $a$ and $b$ of $I$ and $J$ respectively, there is an isomorphism of simplicial sets
$$\Map_{\mathfrak C[I\star^p J]}(a,b)\cong\varnothing.$$
\end{enumerate}
\end{thm}

In particular, like for the classical join, both inclusion of $I$ and $J$ into the weighted join induce fully faithful inclusions after taking the homotopy coherent realizations.

\begin{rmk}
A compact way to describe the content of \cref{mainfactweighted} is that the homotopy coherent realization of the weighted join $I\star^p J$ can be expressed as the pullback
$$\mathfrak C[I\star^p J]\cong\mathfrak C[I\star \Delta[0]]\times_{\mathfrak C[\Delta[0]\star\Delta[0]]}\mathfrak C[\Delta[0]\star^p J]$$
of the canonical maps
$$\mathfrak C[I\star \Delta[0]]\rightarrow\mathfrak C[\Delta[0]\star \Delta[0]]\leftarrow\mathfrak C[\Delta[0]\star J]\leftarrow\mathfrak C[\Delta[0]\star^p J].$$
This is the weighted version of \cite[Theorem 5.3.19]{RV7}.
\end{rmk}

To prove \cref{maintheorem}, we will need to understand the special cases of the formulas from \cref{mainfactweighted} in the case when $I$ is a standard simplex $\Delta[n]$ or the a boundary $\partial\Delta[n]$. The result can be expressed in terms of cubes and their boundaries, which we briefly recall; the precise definitions can be found in \cite[Notation 5.1.6]{RV7}.

\begin{notn}
Let $n\ge0$.
\begin{enumerate}[leftmargin=*]
	\item Denote by $\square[n]$ the $n$-cube, realized as the simplicial set
$$\square[n]:=\Delta[1]^n.$$
	\item Denote by $\partial\square[n]$ the boundary of the $n$-cube, realized as the simplicial set
$$\partial\square[n]:=(\partial\Delta[1]\times\Delta[1]\times\dots\times\Delta[1])\cup(\Delta[1]\times\partial\Delta[1]\times\dots\times\Delta[1])\cup\dots$$
$$\quad\quad\quad\quad\quad\quad\quad\quad\quad\quad\quad\quad\quad\dots\cup(\Delta[1]\times\Delta[1]\times\dots\times\partial\Delta[1])\subset\square[n].$$
	\item For $0\le k\le n$ and $\epsilon=0,1$, denote by $\sqcap^{k}_{\epsilon}[n]$ the $k$-th cubical $n$-horn, realized as the simplicial set
 $$\sqcap^{k}_{\epsilon}[n]:=(\partial\Delta[1]\times\Delta[1]\times\dots\times\Delta[1])\cup\dots\cup(\Delta[1]\times\dots\times\{\epsilon\}\times\dots\times\Delta[1])\cup\dots$$
$$\quad\quad\quad\quad\quad\quad\quad\quad\quad\quad\quad\quad\quad\dots\cup(\Delta[1]\times\Delta[1]\times\dots\times\partial\Delta[1])\subset\partial\square[n].$$
\end{enumerate}
\end{notn}

Recall from \cref{cofibrantreplacement} the explicit model $W^{cof}$ of cofibrant replacement of fibrant weights $W\colon\mathfrak C[J]\to\kan$. The following formulas are a consequence of \cref{mainfactweighted} and the descriptions of the hom-categories of $\mathfrak C[\Delta[n]]$, $\mathfrak C[\partial\Delta[n]]$ and $\mathfrak C[\sqcap^n_0[n]]$, which can be found in \cite[Examples 5.1.7 and 5.1.10]{RV7}.

\begin{prop}
Let $p\colon\tilde J\to J$ be a simplicial map, and $n\ge0$.
\label{computation1}
\begin{enumerate}[leftmargin=*]
\item
For any vertex $x$ of $J$, there are isomorphisms of simplicial sets
$$\Map_{\mathfrak C[\Delta[n]\star^pJ]}(0,x)\cong\square[n]\times W^{cof}(x)\text{ and }\Map_{\mathfrak C[\partial\Delta[n]\star^pJ]}(0,x)\cong\sqcap^{k}_0[n]\times W^{cof}(x).$$
\item
There are isomorphisms of simplicial sets
$$\Map_{\mathfrak C[\Delta[n]\star^pJ]}(0,n)\cong\square[n-1]\text{ and }\Map_{\mathfrak C[\partial\Delta[n]\star^pJ]}(0,n)\cong\sqcap^{k}_0[n-1].$$
\end{enumerate}
\end{prop}

We will also make use of the following technical fact.
\begin{prop}
\label{mainlemma}
If map of simplicial sets $e\colon A\to B$ is a weak equivalence, any lifting problem in $\sset$ of the following form $m\ge0$ admits a solution:
\begin{equation}
\label{LP}
\xymatrix@R=.3cm@C=.7cm{\partial\square[m]\ar[rd]\ar[dd]\ar[rrr]&&&A\ar[rd]^-e&\\
&\sqcap^{m+1}_0[m+1]\ar[dd]\ar[rrr]&&&B.\\
\square[m]\ar[rd]\ar@{-->}@/_1.5pc/[rrruu]&&&&\\
&\square[m+1]\ar@{-->}@/_1.5pc/[rrruu]&&&}
\end{equation}
\end{prop}

\begin{proof}
We first observe that the projection $p_1\colon B\downarrow_{B}e\to A$
is an acyclic Kan fibration, being the pullback of $e\colon A\to B$ along the acyclic Kan fibration $(p_1,p_0)\colon B^{\Delta[1]}\to B$.
Also the inclusion
$$\iota\colon\xymatrix@C=2cm{A\ar[r]^-{(A,const_A,A)}&A\times_A A^{\Delta[1]}\times_A A\ar[r]^-{A\times e^{\Delta[1]}\times e}&A\times_B B^{\Delta[1]}\times_B B=B\downarrow_{B}e}$$
is a weak equivalence, being a right inverse for $p_1$.
Moreover, there is a commutative diagram
$$\xymatrix@R=.3cm{B\downarrow_{B}e\ar[rr]^-{p_0}&&B.\\
&A\ar[ru]_-e\ar[lu]^-{\iota}&}$$
It follows that the map $e$ is a weak equivalence if and only if the fibration
$$p_0\colon B\downarrow_{B}e\to B$$
is an acyclic Kan fibration of simplicial sets. Then $p_0$ has the right lifting property with respect to all cubical boundary inclusions $\partial\square[m]\hookrightarrow\square[m]$ for all $m\ge0$.
Finally, solving the lifting problem for all $m\ge0$:
$$\xymatrix@R=.3cm{\partial\square[m]\ar[r]\ar@{^{(}->}[d]&B\downarrow_{B}e\ar[d]\\
\square[m]\ar[r]\ar@{-->}[ur]&B,}$$
translates to the lifting problem (\ref{LP}),  by evoking \cite[Lemma 6.1.11]{RV7} and choosing $C=D=\Delta[0]$.
\end{proof}

We can now prove the theorem, by showing that the limit of $d$ weighted by $Un(W)$ exists in $\mathfrak N[\cQ]$ if and only if the homotopy limit of $D$ weighted by $W$ exists in $\cQ$, and in this case they are represented by the same object of $\cQ$.

\begin{proof}[Proof of \cref{maintheorem}]
By \cref{weightedlimitquasi-category}, the diagram $d$ admits a weighted limit
$$L\simeq\lim{}^{Un(W)}d$$
if and only if there exists a cone from $L$ over $d$
$$\lambda\colon\Delta[0]\to\mathfrak N[\cQ]^{Un(W)}_{/d}$$
that enjoys the lifting property for $n\ge1$:
$$\xymatrix@R=.3cm{\Delta[0]\ar[r]_-n\ar@/^2pc/[rr]^-{\lambda}&\partial\Delta[n]\ar[r]\ar@{^{(}->}[d]&\mathfrak N[\cQ]^{Un(W)}_{/d}.\\
&\Delta[n]\ar@{-->}@/_1pc/[ru]&}$$

 By \cref{adjunctionweightedjoinslice}, this is equivalent to saying that there exists a map $\lambda$ that fits into a commutative diagram
$$\xymatrix@R=.3cm{J\ar@{^{(}->}[d]\ar[rrd]^-d\\
\Delta[0]\star^{Un(W)}J\ar[rr]|-\lambda&&\mathfrak N[\cQ],\\
\Delta[0]\ar@{_{(}->}[u]\ar[rru]_-L}$$
and enjoys the lifting property for $n\ge1$:
$$\xymatrix@R=.3cm@C=1.7cm{\Delta[0]\star^{Un(W)} J\ar[r]_-{n\star^{Un(W)}J}\ar@/^2pc/[rr]^-{\lambda}&\partial\Delta[n]\star^{Un(W)} J\ar[r]\ar@{^{(}->}[d]&\mathfrak N[\cQ].\\
&\Delta[n]\star^{Un(W)} J\ar@{-->}@/_1pc/[ru]&}$$

Transposing with respect to the adjunction $(\mathfrak C,\mathfrak N)$, this is equivalent to saying that there exists a simplicial functor $\Lambda$ that fits into a commutative diagram
$$\xymatrix@R=.3cm{\mathfrak C[J]\ar@{^{(}->}[d]\ar[rrd]^-D\\
\mathfrak C[\Delta[0]\star^{Un(W)}J]\ar[rr]|-{\Lambda}&&\cQ,\\
\mathfrak C[\Delta[0]]\ar@{_{(}->}[u]\ar[rru]_-L}$$
and enjoys the lifting property for $n\ge1$:
$$\xymatrix@R=.3cm@C=1.7cm{\mathfrak C[\Delta[0]\star^{Un(W)} J]\ar[r]_-{\mathfrak C[n\star^{Un(W)}J]}\ar@/^2pc/[rr]^-{\Lambda}&\mathfrak C[\partial\Delta[n]\star^{Un(W)} J]\ar[r]\ar@{^{(}->}[d]&\cQ.\\
&\mathfrak C[\Delta[n]\star^{Un(W)} J]\ar@{-->}@/_1pc/[ru]&}$$

To solve this problem, it is enough to show that
there exists an object $L$ together with a natural transformation
$$\Lambda\colon W^{cof}=\Map_{\mathfrak C[\Delta[0]\star^{Un(W)}J]}(0,-)\to\Map_{\cQ}(L ,D)$$
that enjoys the further lifting property in $s\Fun(\mathfrak C[J],\sset)$ for all $n\ge1$ and all $A$ object of $\cQ$:\\
\ \\
$\xymatrix@R=.3cm@C=.01cm{\Map_{\mathfrak C[\partial\Delta[n]\star^{Un(W)}J]}(0,n)\ar[rd]\ar[dd]\ar[rr]&&\Map_{\cQ}(A,L)\ar[rd]&\\
&\Map_{\mathfrak C[\partial\Delta[n]\star^{Un(W)}J]}(0,\bullet)\ar[dd]\ar[rr]&&\Map_{\cQ}(A,D\bullet).\\
\Map_{\mathfrak C[\Delta[n]\star^{Un(W)}J]}(0,n)\ar[rd]\ar@{-->}@/_2.5pc/[rruu]&&&\\
&\Map_{\mathfrak C[\Delta[n]\star^{Un(W)}J]}(0,\bullet)\ar@{-->}@/_2.5pc/[rruu]&&}$

Using the explicit formulas from \cref{computation1} and the cofibrant replacement from \cref{cofibrantreplacement}(b),
this lifting problem can be rewritten as
$$\xymatrix@R=.3cm@C=.01cm{\partial\square[n-1]\ar[rd]\ar[dd]\ar[rr]&&\Map_{\cQ}(A,L)\ar[rd]&\\
&\sqcap^{n}_0[n]\times W^{cof}\ar[dd]\ar[rr]&&\Map_{\cQ}(A,D\bullet).\\
\square[n-1]\ar[rd]\ar@{-->}@/_2.5pc/[rruu]&&&\\
&\square[n]\times W^{cof}\ar@{-->}@/_2.5pc/[rruu]&&}$$

We note that $s\Fun(\mathfrak C[J],\sset)$ is tensored and cotensored over $\sset$ (with respect to pointwise tensor and cotensor, as in \cite[\textsection3.8.2]{RiehlCHT}). Using this,
the lifting problem transposes to a lifting property in $\sset$
$$\xymatrix@R=.3cm@C=.01cm{\partial\square[n-1]\ar[rd]\ar[dd]\ar[rr]&&\Map_{\cQ}(A,L )\ar[rd]&\\
&\sqcap^{n}_0[n]\ar[dd]\ar[rr]&&\Map_{s\Fun(\mathfrak C[J],\sset)}(W^{cof},\Map_{\cQ}(A,D\bullet)).\\
\square[n-1]\ar[rd]\ar@{-->}@/_2.5pc/[rruu]&&&\\
&\square[n]\ar@{-->}@/_2pc/[rruu]&&}$$
By \cref{mainlemma}, it is enough to know that the top right map
$$\Map_{\cQ}(A,L)\to\Map_{s\Fun(\mathfrak C[J],\sset)}(W^{cof},\Map_{\cQ}(A,D\bullet))$$
is a weak equivalence of Kan complexes.

Therefore, this is enough to know that there exists an object $L$ of $\cQ$ together with
a (simplicial) natural transformation
$$\Lambda\colon W^{cof}\to\Map_{\cQ}(L,D)$$
that induces a weak equivalence
$$\Map_{\cQ}(A,L)\simeq\Map_{\Fun(\mathfrak C[J],\sset)}(W^{cof},\Map_{\cQ}(A,D\bullet)),$$
namely that the homotopy limit of $D$ weighted by $W^{cof}$ exists and
$$\holim^{W^{cof}}D\simeq L.$$
It follows that
$$\holim^{W}D\simeq\holim^{W^{cof}}D\simeq L\simeq\lim{}^{Un(W)}d,$$
as desired.
\end{proof}

\begin{ex}
If  $\cQ$ is a simplicial category $\cQ$ with simplicial cotensors and pullbacks, for any diagram $D\colon\mathfrak C[N\Gamma]\to\cQ$ with image
$$\xymatrix{A&\ar[l] _-FB\ar[r]^-G&C}$$
the comma construction $f\downarrow_B g$ defines a pullback of the adjoint diagram $d\colon N\Gamma\to\mathfrak N[\cQ]$ in the homotopy coherent nerve $\mathfrak N[\cQ]$. Indeed,
$$F\downarrow_B G:=\lim{}^{W^{comma}}D\simeq\lim{}^{\Delta[0]_{\Gamma}^{cof}}D\simeq\lim{}^{Un(\Delta[0]_{\Gamma})}d=\lim d=:A\times_{B} C.$$
\end{ex}

We can now prove \cref{allweightedhomototopylimitsareconstantenriched}.

\begin{proof}[Proof of \cref{allweightedhomototopylimitsareconstantenriched}]
Combining \cref{maintheorem,allweightedhomototopylimitsareconical}, we obtain the equivalences
\[\begin{array}{rcl}
\holim{}^WD&\simeq&\holim{}^{W^{cof}}D\\
&\simeq&\lim{}^{Un(W)}d\\
&\simeq&\lim{}^{\Delta[0]}\left(Un(W)\to J\stackrel{D}{\longrightarrow}\mathfrak N[\cQ]\right)\\
&\simeq&\holim^{\Delta[0]}\left(\mathfrak C[Un(W)]\to\mathfrak C[J]\stackrel{D}{\longrightarrow}\cQ\right),\\
\end{array}\]
as desired.
\end{proof}

\appendix \label{appendix}
\stepcounter{section}
\section*{Appendix: Necklaces in a weighted join} 
%%%\renewcommand{\thethm}{\Alph{section}.\arabic{thm}}
%%%\renewcommand{\thedefn}{\Alph{section}.\arabic{thm}}
%%%\renewcommand{\theprop}{\Alph{section}.\arabic{thm}}
%%%\renewcommand{\thermk}{\Alph{section}.\arabic{thm}}

%\appendix 
%\section{Necklaces in a weighted join} 

\label{appendix}

In this section, we give a description of the hom-categories of a homotopy coherent realization $\mathfrak C[J]$ in terms of flagged necklaces, and then specialize the result to understand $\mathfrak C[I\star^{p}J]$.

\begin{notn}
Given $\vec n:=(n_1,\dots,n_k)$, denote by $\Delta[\vec n]$
the head-to-tail wedge of standard simplices
$$\Delta[\vec n]:=\Delta[n_1]\vee\dots\vee\Delta[n_k].$$
\end{notn}

We slightly revisit the standard terminology of necklaces from \cite[\textsection1.1]{ds} and \cite[\textsection2]{riehlnecklace}.
\begin{defn}
Let $J$ be a simplicial set.
\begin{itemize}[leftmargin=*]
\item A \emph{necklace} in $J$ consists of a map $\tau\colon\Delta[\vec n]\to J$. The necklace is \emph{totally non-degenerate} if each \emph{bead} $\tau|_{\Delta[n_k]}\colon\Delta[n_k]\to J$ represents a non degenerate simplex of $J$. 
We denote by $V(\tau):=\tau(\Delta[\vec n])_0$ the set of all vertices of $\tau(\Delta[\vec n])\subset J$
and by $J(\tau)$
the subset of all the vertices of $\tau(\Delta[\vec n])$ coming from vertices of $\Delta[\vec n]$ along which the beads of $\tau$ have been glued.
\item A \emph{flag} of degree $m$ for the necklace $\tau\colon\Delta[\vec n]\to J$ consists of a sequence of nested sets $(T_i)_{i=0}^m$ such that
$$J(\tau)=T_0\subset\dots\subset T_m=V(\tau).$$
\item  A \emph{flagged (totally non-degenerate) necklace} of degree $m$ consists of a pair $\left(\tau,(T_i)_{i=0}^m\right)$ where $\tau$ is a (totally non-degenerate) necklace and $(T_i)_i$ is a flag of degree $m$ for $\tau$. In particular, a flagged (totally non-degenerate) necklace of degree $0$ consists of just of a (totally non-degenerate) unflagged necklace $\tau$.
\end{itemize}
\end{defn}

\begin{rmk}
The collection $\Nec(J)_m$ of flagged necklaces in $J$ of degree $m$ is functorial in $J$.
\end{rmk}

\begin{rmk}[{\cite[\textsection1.1]{ds}, \cite[\textsection2]{riehlnecklace}}]
Any flagged necklace in $J$ has a source and a target. If $\Nec(J)_m(x,x')$ denotes the collection of flagged totally non-degenerate necklaces of degree $m$ in $J$ with source $x$ and target $x'$, there is a well-defined concatenation
$$[-,-]\colon \Nec(J)_m(x,x')\times \Nec(J)_m(x',x'')\to \Nec(J)_m(x,x''),$$
given by the assignment
$$\left((\tau\colon\Delta[\vec n]\to J,(T_i)_{i=0}^m),(\tau'\colon\Delta[\vec n']\to J,(T'_i)_{i=0}^m)\right)\mapsto([\tau,\tau']\colon\Delta[\vec n]\vee\Delta[\vec n']\to J,(T_i\cup T_i')_{i=0}^m).$$
\end{rmk}

The following result from \cite{riehlnecklace} records how necklaces in $J$ describe morphisms of $\mathfrak C[J]$, and is a variant of the characterization originally due to Dugger-Spivak \cite{ds}.

\begin{thm}[{\cite[Theorem 2.1]{riehlnecklace}}]
\label{thmnecklaces}
For any simplicial set $J$ and $n\ge0$, there is a bijection
$$\Map_{\mathfrak C[J]}(x,x')_m\cong \Nec(J)_m(x,x')$$
which is natural in $J$.
\end{thm}

Aiming to understand the morphisms of $\mathfrak C[I\star^p J]$, we look at necklaces in $I\star^p J$.
Intuitively, a necklace in a weighted join must consist of three (possibly empty) parts: a necklace in $I$, followed by 
a single simplex in $I\star\tilde J$, followed by a necklace in $J$. 
Pushing the investigation further, by nature of the join construction the simplex in the middle must consist of a simplex of $I$ and a simplex of $\tilde J$.

The following proposition makes this precise.

\begin{prop}
\label{necklaceinweightedjoin}
Let $p\colon\tilde J\to J$ be a map. Any flagged totally non-degenerate necklace of degree $m$ in a fat join $I\star^p J$
$$(\tau\colon\Delta[\vec n]\to I\star^p J,(T_i)_{i=0}^m)$$
can be uniquely written as a concatenation of three consecutive flagged necklaces of the following form:
\begin{enumerate}[leftmargin=*]
\item a flagged totally non-degenerate necklace of degree $m$ in $I$
$$(\tau^I\colon\Delta[\vec n^I]\to I,(T^I_i)_{i=0}^m)$$
whose source and target are vertices of $I$.
\item a flagged non-degenerate bead of degree $m$ in $I\star^p\tilde J$ of the form
$$(\sigma^I\star\sigma^{\tilde J}\colon\Delta[N^I+N^J]\cong\Delta[N^I]\star\Delta[N^b]\to I\star\tilde J,(S_i)_{i=0}^m)$$
whose source is a vertex of $I$ and whose target is a vertex of $\tilde J$.
\item a flagged totally non-degenerate necklace of degree $m$ in $J$.
$$(\tau^J\colon\Delta[\vec n^J]\to J,(T^J_i)_{i=0}^m)$$
whose source and target are vertices of $J$.
\end{enumerate}
\end{prop}

\begin{proof}
First let
$$\overline k:=\max\{k\ |\ \tau(\Delta[n_k])\subset I\}.$$
Note then that the bead
$$\tau|_{\Delta[\overline k+1]}\colon\Delta[\overline k+1]\to I\star^p J$$
defines a simplex in $I\star^p J$ whose source is a vertex of $I$ and whose target is a vertex of $\tilde J$. 
By unpacking the definition of the weighted join $I\star^p J$, the collection of $(\overline k+1)$-simplices of $I\star^pJ$ is given by
$$\begin{array}{rcl}
(I\star^p J)_{\overline k+1}&:=&\left(I_{\overline k+1}\amalg\coprod_{i=1}^{\overline k}(I_{i}\times\tilde J_{\overline k-i})\amalg\tilde J_{\overline k+1}\right)\amalg_{\tilde J_{\overline k+1}}J_{\overline k+1}\\
&\cong& I_{\overline k+1}\amalg\coprod_{i=1}^{\overline k}(I_{i}\times\tilde J_{\overline k-i})\amalg J_{\overline k+1}.
\end{array}$$
Thus the simplex $\tau|_{\Delta[\overline k+1]}$, which is not contained entirely in $I$ nor $\tilde J$, can be uniquely written as
$$\sigma^I\star\sigma^{\tilde J}\colon\Delta[\overline k+1]\cong\Delta[i]\star\Delta[\overline k-i]\to I\star\tilde J,$$
where $\sigma^I\colon\Delta[i]\to I$ is a simplex of $I$, and $\sigma^{\tilde J}\colon\Delta[\overline k-i]\to\tilde J$ is a simplex of $\tilde J$. We now build the three pieces.
\begin{enumerate}[leftmargin=*]
\item If we set $\vec n^I:=(n_1,\dots,n_{\overline k})$, we get a necklace in $I$
$$\tau^I:=\tau|_{\Delta[\vec n^I]}\colon\Delta[\vec n^I]\to I,$$
endowed with the flag $T_i^I:=T_i\cap V(I)$.
\item If we set $N^I:=i$ and $N^J:= \overline k-i$, we get a simplex in $I\star\tilde J$
$$\sigma^I\star\sigma^{\tilde J}=\tau|_{\Delta[\overline k]}\colon\Delta[\overline k+1]\cong\Delta[i]\star\Delta[\overline k-i]\to I\star\tilde J,$$
endowed with the flag $S_i:=T_i\cap V(\tau|_{\Delta[n_{\overline k}]})$.
\item If we set $\vec n^J:=(n_{\overline k+2},\dots,n_K)$, we get a necklace in $J$
$$\tau^J:=\tau|_{\Delta[\vec n^J]}\colon\Delta[\vec n^J]\to J,$$
endowed with the flag $T_i^J:=T_i\cap V(J)$.
\end{enumerate}
This proves the existence of the desired decomposition. For uniqueness, we observe that the assignment is in fact invertible and there is no loss of information.
Indeed, each triple of concatenable flagged necklaces $(\tau^I,\sigma^I\star\sigma^{\tilde J},\tau^J)$ of the form (1), (2) and (3) (as in the statement of the proposition) can be concatenated to form a flagged necklace $[\tau^I,\sigma^I\star\sigma^{\tilde J},\tau^J]$ in $I\star^p J$. The operation
$$(\tau^I,\sigma^I\star\sigma^J,\tau^J)\mapsto[\tau^I,\sigma^I\star\sigma^{\tilde J},\tau^J]$$
can be checked to be an inverse for the decomposition
$$\tau\mapsto(\tau^I,\sigma^I\star\sigma^{\tilde J},\tau^J)$$
that we described in the first part of this proof.
\end{proof}

Specializing to the case $I:=\Delta[0]$, we obtain the following.

\begin{prop}
\label{necklaceinfatjoinwithDelta0}
Let $p\colon\tilde J\to J$ be a map. Any flagged totally non-degenerate necklace of degree $m$ in a fat join $\Delta[0]\star^p J$
$$(\tau\colon\Delta[\vec n]\to \Delta[0]\star^p J,(T_i)_{i=0}^m)$$
can be uniquely written as a concatenation of two flagged consecutive necklaces of the following form:
\begin{enumerate}[leftmargin=*]
\item a flagged non-degenerate bead of degree $m$ in $\Delta[0]\star\tilde J$ of the form
$$(\top\star\sigma^{\tilde J}\colon\Delta[1+N]\cong\Delta[0]\star\Delta[N]\to\Delta[0]\star\tilde J,(S_i)_{i=0}^m)$$
from the cone boint $\top$ and to a vertex of $\tilde J$;
\item a flagged totally non-degenerate necklace of degree $m$ in $J$.
$$(\tau^J\colon\Delta[\vec n^J]\to J,(T^J_i)_{i=0}^m)$$
whose source and target are vertices of $J$.
\end{enumerate}
\end{prop}

We are finally ready to prove \cref{mainfactweighted}, which describes the morphisms of $\mathfrak C[I\star^b J]$ from a vertex $a$ of $I$ to a vertex $b$ of $J$.

\begin{proof}[Proof of \cref{mainfactweighted}]
In degree $m$, consider the map
$$\Phi_0\colon\Map_{\mathfrak C[I\star^p J]}(a,b)_0\to\Map_{\mathfrak C[I\star\Delta[0]]}(a,\top)_0\times\Map_{\mathfrak C[\Delta[0]\star^pJ]}(\bot,b)_0,$$
which, modulo the identifications from \cref{thmnecklaces,necklaceinweightedjoin,necklaceinfatjoinwithDelta0}, is described by
$$\tau=[\tau^I,\sigma^I\star \sigma^{\tilde J},\tau^J]\mapsto\left([\tau^I,\sigma^I\star \bot],[\top\star \sigma^{\tilde J},\tau^J]\right).$$
The map $\Phi_0$ is bijective. Indeed, the map
$$\Psi_0\colon\Map_{\mathfrak C[I\star\Delta[0]]}(a,\top)_0\times\Map_{\mathfrak C[\Delta[0]\star^pJ]}(\bot,b)_0\to\Map_{\mathfrak C[I\star^pJ]}(a,b)_0,$$
described, again using the identifications from \cref{thmnecklaces,necklaceinweightedjoin,necklaceinfatjoinwithDelta0}, by the assignment
$$(\nu^I,\nu^J)=\left([\tau^I,\sigma^I\star \top],[\bot\star\sigma^{\tilde J},\tau^J]\right)\mapsto[\tau^I,\sigma^I\star\sigma^{\tilde J},\tau^J]$$
can be checked to be the inverse function.

By taking into account the structure of flags suitably, the assignment that defines $\Phi_0$ can be promoted to a well-defined function
$$\Phi_m\colon\Map_{\mathfrak C[I\star^pJ]}(a,b)_m\to\Map_{\mathfrak C[I\star\top]}(a,\top)_m\times\Map_{\mathfrak C[\bot\star J]}(\bot,b)_m.$$
If $(T_i)_i$ denotes the flag of $\tau$, the flags for $\tau^I$ and $\tau^J$ are described in the proof of \cref{necklaceinweightedjoin}, and we declare that the flags for $\sigma^I\star\bot$ and $\top\star\sigma^{\tilde J}$ are given for $i=0,\dots,m$ by
$$\left(T_i\cap V(\sigma^I\star\sigma^{\tilde J})\cap V(I)\right)\cup\{\top\}\quad\text{ and }\quad\left(T_i\cap V(\sigma^I\star\sigma^{\tilde J})\cap V(\tilde J)\right)\amalg_{V(\tilde J)}V(J)\cup\{\bot\}.$$

The map $\Phi_m$ is bijective. Indeed, the function $\Psi_0=\Phi_0^{-1}$ can also be be promoted to a well-defined function
$$\Psi_m\Map_{\mathfrak C[I\star^pJ]}(a,b)_m\to\Map_{\mathfrak C[I\star\top]}(a,\top)_m\times\Map_{\mathfrak C[\bot\star J]}(\bot,b)_m$$
by taking into account the flags as follows.
If $(S^I_i)_i$ and $(S^J_i)_i$ are the flags of $\nu^I$ and $\nu^J$, the flag of $\sigma^I\star\sigma^{\tilde J}$ is declared to be given for $i=0,\dots,m$ by
$$(S^I_i\setminus\{\top\})\cup(S^J_i\setminus\{\bot\})\amalg_{V(\tilde J)}V(J).$$

Finally, a direct inspection should show the components of $\Phi_m$,
$$\Map_{\mathfrak C[I\star^pJ]}(a,b)_m\to\Map_{\mathfrak C[I\star\Delta[0]]}(a,\top)_m\text{ and }\Map_{\mathfrak C[I\star^pJ]}(a,b)_m\to
\Map_{\mathfrak C[\Delta[0]\star^pJ]}(\bot,b)_m,$$
is compatible with the simplicial structure, yielding the desired simplicial isomorphism.
\end{proof}

%%%
%%%
%%%

%%%%%%%%%%%%%%%%%%%%%%%%%%%%%%%%%%%%%%%%%%%%%%%%%%%%%%%%%%%%%%%%%%%%%%%%%%%%%%%%%%%%%%%%%%%%%%%%%%%%%%%%%%%%%%%%%%%
%\bibliographystyle{abbrv}

% \addtocontents{toc}{\protect\setcounter{tocdepth}{0}}
\bibliographystyle{amsalpha}
\bibliography{ref}

\end{document}